\documentclass[preprint,12pt]{elsarticle}
\usepackage[left=2.7cm, right=2.7cm, top=3.2cm, bottom=3.2cm]{geometry}

\usepackage{graphicx}
\usepackage{epsfig}
\usepackage{amsmath}
\usepackage{hyperref}
\usepackage{amsthm}
\usepackage{xpatch}

\xpatchcmd{\proof}{\itshape}{\prooflabelfont}{}{}
\newcommand{\prooflabelfont}{\bfseries}

\theoremstyle{definition}
\newtheorem{definition}{Definition}
\newtheorem{remark}{Remark}

\theoremstyle{plain}
\newtheorem{theorem}{Theorem}
\newtheorem{lemma}{Lemma}
\newtheorem{proposition}{Proposition}

\usepackage{setspace}

\usepackage{listings}
\usepackage{algorithm,algpseudocode}
\usepackage{setspace} 
\newcommand{\var}{\texttt}
\newcommand{\func}{\textsc}
\usepackage{blkarray}
\usepackage{array}

\usepackage{amssymb}

\journal{Applied and Computational Harmonic Analysis}

\begin{document}

\begin{frontmatter}



\title{High-Order Synchrosqueezed Chirplet Transforms for Multicomponent Signal Analysis}


\author[inst1,inst2]{Yi-Ju Yen}

\affiliation[inst1]{organization={Graduate Institute of Communication Engineering, National Taiwan University},
            addressline={No. 1, Sec. 4, Roosevelt Rd.}, 
            city={Taipei},
            postcode={10617}, 
            country={Taiwan (R.O.C.)}}

\author[inst1]{De-Yan Lu}
\affiliation[inst2]{organization={Graduate Institute of Mathematics, National Taiwan University},
            addressline={No. 1, Sec. 4, Roosevelt Rd.}, 
            city={Taipei},
            postcode={10617}, 
            country={Taiwan (R.O.C.)}}

\author[inst3]{Sing-Yuan Yeh}

\affiliation[inst3]{organization={Data Science Degree Program, National Taiwan University and Academia Sinica},
            addressline={No. 1, Sec. 4, Roosevelt Rd.}, 
            city={Taipei},
            postcode={10617}, 
            country={Taiwan (R.O.C.)}}
\author[inst1]{Jian-Jiun Ding\corref{cor1}}
\cortext[cor1]{Corresponding author.}
\author[inst2]{Chun-Yen Shen}
\begin{abstract}
This study focuses on the analysis of signals containing multiple components with crossover instantaneous frequencies (IF). This problem was initially solved with the chirplet transform (CT). Also, it can be sharpened by adding the synchrosqueezing step, which is called the synchrosqueezed chirplet transform (SCT). However, we found that the SCT goes wrong with the high chirp modulation signal due to the wrong estimation of the IF. In this paper, we present the improvement of the post-transformation of the CT. The main goal of this paper is to amend the estimation introduced in the SCT and carry out the high-order synchrosqueezed chirplet transform. The proposed method reduces the wrong estimation when facing a stronger variety of chirp-modulated multi-component signals. The theoretical analysis of the new reassignment ingredient is provided. Numerical experiments on some synthetic signals are presented to verify the effectiveness of the proposed high-order SCT.
\end{abstract}




\end{frontmatter}


\section{Introduction}
\label{sec:introduction}
Time series analysis is getting increasingly important in many fields, including radar system analysis, seismic wave detection, medical signals, etc. Due to the non-stationarity of the systems, the observed signals are always oscillatory and in the form of multiple components. 
Many time-frequency (TF) analysis tools have been invented to decompose and clarify them. The reassignment method(RM) manages to sharpen the representation. Empirical mode decomposition (EMD), the short-time Fourier transform (STFT), the STFT-based synchrosqueezing transform (FSST), or its higher order version, like the vertical second-order synchrosqueezing(VSST) and the $n$-th order STFT-based SST(FSSTn) was developed to retrieve the signals component by component. However, these methods require some specific separation conditions. To decompose a signal properly is an important issue before we start to analyze them.
However, the reassignment method(RM)\cite{auger1995improving, fitz2009unified} has the problem of invertibility. One cannot retrieve the mode information after performing the RM. EMD\cite{rilling2003empirical, 7080206} and the short time Fourier transform(STFT)\cite{cohen1995time} for multi-component signal requires the ``separation condition", that is to say,
if a signal $f(t)$ has the form as
\begin{equation}
    f(t) = \sum_{k=1}^{K}f_k(t)= \sum_{k=1}^{K}A_k(t)e^{2\pi i \phi_k(t)}
\end{equation}
where $A_k(t)$ is the non-negative function which is known as the amplitude function and $\phi_k(t)$ is the phase function with $\phi'_k(t)>0$, then STFT requires that, for any $k\neq j$, one should have $\left|\phi'_k(t)-\phi'_j(t)\right|\geq\Delta$ for some constant $\Delta$ depending on the resolution\cite{9970392} and for all $t$. This condition is understandable from the spectrogram graph since $\Delta$ stands for the gap between the bright bands\cite{cohen1995time, nelson2002instantaneous}. The FSST\cite{meignen2012new} or its higher order version(VSST, FSSTn)\cite{oberlin2015second, pham2017high} can improve the resolution of the time-frequency representation. However, they still require the separation condition to make the approximate reconstruction formula feasible \cite{behera2018theoretical}. In most conditions, the signal does not satisfy the separation condition and its time-frequency (TF) analysis will be very challenging. Due to the overlapping instantaneous frequency (IF) problem, we consider the $3$ dimensional time-frequency-chirp (TFC) representation when the modes have different chirp rates at the crossing points of the IF. With this technique, two crossing IFs in the TFC space are curved vertically and no longer overlap in the space \cite{mann1992time, mann1995chirplet, zhu2020frequency}. For other related and interesting methodologies and results, one can look at the paper \cite{boashash1994polynomial, wu2020current, li2020if, oberlin2017second}, and the references contained therein.

In this paper, we try to extend the method in \cite{chen2023disentangling} to the signals containing strong chirp modulation. The synchrosqueezed chirplet transform(SCT) proposed by Chen and Wu weakens the constraint of signal decomposition. In \cite{chen2023disentangling}, the target signal should satisfy $\left|\phi_{k}^{\prime}(t)-\phi_{l}^{\prime}(t)\right|+\left|\phi_{k}^{\prime \prime}(t)-\phi_{l}^{\prime \prime}(t)\right| \geq 2 \Delta$ for all $k\neq l$ and each term $f_k(t) = A_k(t)e^{2\pi i \phi_k(t)}$ is an $\epsilon$--intrinsic chirp type ($\epsilon$--ICT) function, i.e., it is ``similar" to a chirp function for all $k$. In other words, they assume that $\phi'''_k(t)$ are negligible, but, in many cases, the signals contain higher-order chirp modulation, which makes the reassignment ingredients used in \cite{chen2023disentangling} biased. This work proposes to give the third-order estimation of the phase function to adjust the estimator we use in the synchrosqueezing transform and improve the performance of the existing SCT with the new ingredients.

The outline of the paper is as follows: in Section \ref{sec:pre}, we recall some definitions and notations commonly used in TF analysis. We then present our proposed method and some properties in Section \ref{sec:main}. In Section \ref{sec:results}, we provide some numerical analysis of the synthetic data with our new method and compare the results with the SCT.

\section{Preliminary: Background of the SST and the SCT.}
\label{sec:pre}
In the beginning, we intuitively define  the AM-FM multi-component signal(MCS) as the superposition of the AM-FM signal as follows, which will be intensively studied later, 
\begin{equation}
f(t) = \sum _{k = 1}^{K}f_k(t)= \sum _{k = 1}^{K}A_k(t)e^{2\pi i \phi_k(t)}
\end{equation}
where $K\in \mathbb{N}$ is a known number from background knowledge. Each $f_k(t)$ is called a mode. Such a signal is completely described by the ideal TF representation.
\begin{definition}[Ideal Time-Frequency (TF) Representation]
Fixed an AM-FM multi-component signal $f(t) = \sum _{k = 1}^{K}f_k(t)= \sum _{k = 1}^{K}A_k(t)e^{2\pi i \phi_k(t)}$ , the ideal TF representation of $f(t)$ is given by
\begin{equation}
\mathrm{TI}_f(t,\omega) = \sum _{k = 1}^{K}A_k(t)\delta(\omega-\phi'_k(t))
\end{equation}
where $\delta$ denotes the Dirac distribution. 
\end{definition}
\begin{definition}[Short time Fourier transform (STFT)]
Given a signal $f \in L^2(\mathbb{R})$ and a window function $g\in \mathcal{S}(\mathbb{R})$, where the $\mathcal{S}(\mathbb{R})$ is the Schwartz function on $\mathbb{R}$ and the modified STFT of $f$ is defined by 
\begin{equation}
V^{(g)}_f(t, \xi) = \int_{\mathbb{R}}f(x)g^*(x-t)e^{-2\pi i \xi (x-t)}dx
\end{equation}
where $g^*$ is the complex conjugate of $g$ and $|V^{(g)}_f(t, \xi)|^2$ is the spectrogram.
\end{definition}
The key idea of synchrosqueezing is to sharpen the blurred TF distribution by using some instantaneous frequency estimator to reassign the value at time $t$ and frequency $\xi$ vertically.
The STFT-Based SST(FSST) \cite{behera2018theoretical} is defined by 
\begin{equation}
T^{g,\gamma}_{f}(t, \omega) = \frac{1}{g^*(0)}\int_{\{\xi:
| V^{(g)}_f(t,\xi) | > \gamma\}} V^{(g)}_f(t,\xi)\delta(\omega- \hat{\omega}_f(t, \xi))d\xi
\end{equation}
where 
$\tilde{\omega}_f(t,\xi) = \frac{1}{2\pi i}\frac{\partial_tV^{(g)}_f(t,\xi)}{V^{(g)}_f(t,\xi)}$ and $\hat{\omega}_f(t, \xi):= Re\left(\tilde{\omega}_f(t,\xi)\right) $. $\tilde{\omega}_f(t,\xi)$ is often called the first order modulation-based ridge detector which is used to estimate the first derivative of phase function. 
\begin{definition}[The second-order synchrosqueezing transform (VSST)]
Set $f\in L^2(\mathbb{R})$, define 
    $$\tilde{t}_f(t, \xi)=t-\frac{\partial_\eta V_f^g(t, \xi)}{i 2 \pi V_f^g(t, \xi)}\text{. }$$
When $V_f^g(t, \xi) \neq 0$ and $\partial_t\left(\frac{\partial_\eta V_f^g(t, \xi)}{V_f^g(t, \xi)}\right) \neq i 2 \pi$ ,we also define 
    $$\tilde{q}_f(t, \xi)=\frac{\partial_t \tilde{\omega}_f(t, \xi)}{\partial_t \tilde{t}_f(t, \xi)}=\frac{\partial_t\left(\frac{\partial_t V_f^g(t, \xi)}{V_f^g(t, \xi)}\right)}{i 2 \pi-\partial_t\left(\frac{\partial_\eta V_f^g(t, \xi)}{V_f^g(t, \xi)}\right)}$$
The second-order synchrosqueezing transform (VSST) \cite{oberlin2015second} is further defined by 
$$
T^{g,\gamma}_{f}(t, \omega) = \frac{1}{g^*(0)}\int_{\{\xi:
| V^{(g)}_f(t,\xi) | > \gamma\}} V^{(g)}_f(t,\xi)\delta(\omega- \hat{\omega}^{(2)}_f(t, \xi))d\xi
$$ where
$$
\begin{aligned}
    \tilde{\omega}_f^{(2)}(t, \xi)=\left\{\begin{array}{cc}
\tilde{\omega}_f(t, \xi)+\tilde{q}_f(t, \xi)\left(t-\tilde{t}_f(t, \xi)\right) & \text { if } \partial_t \tilde{t}_f(t, \xi) \neq 0 \\
\tilde{\omega}_f(t, \xi) & \text { otherwise }
\end{array}\right.
\end{aligned}
$$
and $\hat{\omega}^{[2]}_f(t, \xi) = Re\left(\tilde{\omega}_f^{(2)}(t, \xi)\right)$.
\end{definition}
It was shown in \cite{oberlin2015second} that $Re\left\{\tilde{q}_{t, f}(t, \eta)\right\}=\phi^{\prime \prime}(t)$ when $f(t) = A(t) e^{i 2 \pi \phi(t)}$ where both $\log (A(t))$ and $\phi(t)$ are quadratic. Also, $Re\left\{\tilde{\omega}_{t, f}^{(2)}(t, \eta)\right\}$ is an exact estimation of $\phi^{\prime}(t)$ for this kind of signals. With the thinking of the Taylor expansion, one can always pursue the higher-order estimation, but the $2$-dimensional version is not enough for the overlapping instantaneous frequencies. 

Here, we start to add the information in the chirp axis. we introduce a technique to extract more information, not only in the TF plane but also in the time-frequency-chirp (TFC) space. Like the $2$-dimensional case, the ideal TFC representation in the $ 3$-dimensional space can also completely describe a signal. Fixed an AM-FM multi-component signal $f(t) = \sum _{k = 1}^{K}f_k(t)= \sum _{k = 1}^{K}A_k(t)e^{2\pi i \phi_k(t)}$, the ideal TFC representation in the TFC space of $f(t)$ is given by
$$ \mathrm{TI}_f(t,\xi, \lambda) = \sum _{k = 1}^{K}A_k(t)\delta(\xi-\phi'_k(t))\delta(\lambda-\phi''_k(t))$$
where $\delta$ denotes the Dirac distribution. 
\begin{definition}[Chirplet transforms (CT)  \cite{mann1992time, mann1995chirplet}]\label{CT}
The chirplet transform of $f\in L^2(\mathbb{R})$ given a window function $g \in \mathcal{S}(\mathbb{R})$ at $(t, \xi, \lambda)$ is defined as
$$
T_f^{(g)}(t,\xi,\lambda)= \int_{\mathbb{R}}f(x)g^{*}(x-t)e^{-i 2\pi \xi (x-t)- i \pi \lambda(x-t)^2}dx
$$
where 
$t\in\mathbb{R}$ is time, $\xi \in\mathbb{R}$ means frequency, and $\lambda\in\mathbb{R}$ indicates the chirp rate. Since $g$ is Schwartz, $T_{f}^{(g)}(t, \xi, \lambda)$ is well defined at each $(t, \xi, \lambda)$, and it is a complex function on the TFC domain. 
\end{definition}
We adopted Lemma 1 in \cite{chen2023disentangling} which set $f(x) = e^{i 2\pi \xi_0 + i \pi \lambda_0 x^2}$. In this case, the regularity of $T_f^{(g)}(t,\xi,\lambda)$ is about $\frac{1}{\sqrt{|\lambda_0 - \lambda|}}$ as $\xi $ is sufficient close to $\xi_0+\lambda_0 t$. Since the regularity is poor and decays slowly with $\lambda$, it is reasonable to apply the synchrosqueezing transform to make the energy of $T^{(g)}_f(t, \xi, \lambda)$ more concentrated \cite{daubechies2011synchrosqueezed}.
\begin{definition}[Synchrosqueezed Chirplet Transform (SCT) \cite{chen2023disentangling}]\label{SCT}
Take $f \in L^2(\mathbb{R})$ as the input and $g \in S(\mathbb{R})$
as the window. For a small $\alpha > 0$, define the synchrosqueezed chirplet transform
(SCT) with resolution $\alpha$ as
    \begin{align*}
    &S_{f}^{(g, \alpha)}(t, \xi, \lambda):=
    \iint_{\mathbb{R}^{2}} T_{f}^{(g)}(t, \eta, \gamma) h_{\alpha}(\xi-\omega_{f}^{(g)}(t, \eta, \gamma)) h_{\alpha}(\lambda-\mu_{f}^{(g)}(t, \eta, \gamma)) \mathrm{d} \eta \mathrm{d} \gamma
    \end{align*}
where $\mu_{f}^{(g)}(t, \xi, \lambda)$ and $\omega_{f}^{(g)}(t, \xi, \lambda)$ are defined as 
$$
\mu_{f}^{(g)}(t, \xi, \lambda) =\frac{\partial_t \left(\frac{\partial_t T^{(g)}_f }{T^{(g)}_f }\right)}{2\pi i\left(1+\partial_t\left(\frac{T^{(tg)}_f}{T^{(g)}_f}\right)\right)}, $$
$$\omega_{f}^{(g)}(t, \xi, \lambda) = \frac{1}{2\pi i}\left(\frac{\partial_t T^{(g)}_f }{T^{(g)}_f }\right)-\mu_f^{(g)}(t, \xi, \lambda)\left(\frac{T^{(tg)}_f}{T^{(g)}_f}\right)
$$
and $h_{\alpha}$ is an “approximate $\delta$-function" 
(i.e. $h(x)$ is smooth, 
$\int_{\mathbb{R}} h(x) dx=1$ and $h_{\alpha}(x):=\frac{1}{\alpha}h(\frac{x}{\alpha})$).
\end{definition}

\clearpage

\section{Proposed Method}
\label{sec:main}
Note that the SCT in Definition \ref{SCT} can estimate the $2^{nd}$ order derivative of the phase function. Now, we try to perform higher-order estimation. We intend to find the estimation of the first three derivatives of the phase function $\phi(t)$. Assume that the signal $f(t)$ is of the polynomial phase function form, i.e.,
\begin{align*}
    f(x) = \exp{(2\pi i (\frac{\theta_0}{3!} x^3 + \frac{\mu_0}{2!} x^2 + \omega_0 x))},
\end{align*}
where 
\begin{equation}\label{eq:taylor_of_phi(t)}
\begin{aligned}
    \phi(x) &= \frac{\theta_0}{3!} x^3 + \frac{\mu_0}{2!} x^2 + \omega_0 x\\
            &=\phi(t) + \frac{\phi'(t)}{1!}(x-t)+\frac{\phi''(t)}{2!}(x-t)^2 +\frac{\phi'''(t)}{3!}(x-t)^3.
\end{aligned}
\end{equation}

It is worth to notice that $\partial_t T_f^{(g)} = T_{f'}^{(g)}$ and $f'(t) = 2\pi i \phi'(t)f(t)$. Therefore, from \eqref{eq:taylor_of_phi(t)}, we have 
\begin{equation}
\begin{aligned}
    f'(x) &= 2\pi i \phi'(x)f(x) \\
          &= 2\pi i \left[\phi'(t) + \frac{\phi''(t)}{1!}(x-t)+\frac{\phi'''(t)}{2!}(x-t)^2\right]f(x),
\end{aligned}
\end{equation}
which lead to 
\begin{equation}
\begin{aligned}
    \partial_t T_f^{(g)} &= i 2\pi \left[\phi'(t) T_f^{(g)} + \frac{\phi''(t)}{1!} T_f^{(tg)}+\frac{\phi'''(t)}{2!} T_f^{(t^2g)}\right]\\
    &=i 2 \pi \left([\omega_0+\mu_0 t+\frac{\theta_0}{2 !} t^2] T_f^{(g)}+\left[\mu_0+\theta_0 t\right] T_f^{(t g)}+ \frac{1}{2}\theta_0 T_f^{\left(t^2 g\right)}\right)\\
    \partial_t\left(\frac{\partial_t T_f^{(g)}}{T_f^{(g)}}\right)  &= \partial_t\left(2\pi i[\omega_0+ \mu_0 t+ \frac{\theta_0}{2!}t^2]+ i 2\pi \left[\mu_0+\theta_0 t\right]\left(\frac{T_f^{(tg)}}{T_f^{(g)}}\right)\right.\\
    &\hspace{1.3cm}\left.+\pi i\theta_0 \left(\frac{T_f^{(t^2g)}}{T_f^{(g)}}\right)\right)\\
    &= i 2\pi \left[\mu_0 + \theta_0t\right]+ i 2\pi \theta_0 \frac{T_f^{(tg)}}{T_f^{(g)}}+2\pi i\left[\mu_0+\theta_0 t\right]\partial_t\left(\frac{T_f^{(tg)}}{T_f^{(g)}}\right)\\
    &\quad+\pi i\theta_0 \partial_t\left(\frac{T_f^{(t^2g)}}{T_f^{(g)}}\right),\\
    \frac{\partial_t \left(\frac{\partial_t T^{(g)}_f }{T^{(g)}_f }\right) }{1+\partial_t\left(\frac{T^{(tg)}_f}{T^{(g)}_f}\right)} &= i 2\pi \left[\mu_0 + \theta_0t\right]+ \frac{i \pi \theta_0}{1+\partial_t\left(\frac{T^{(tg)}_f}{T^{(g)}_f}\right)}\left[ 2\frac{T^{(tg)}_f}{T^{(g)}_f}+\partial_t\left(\frac{T^{(t^2g)}_f}{T^{(g)}_f}\right)\right].
\end{aligned}
\end{equation}
Take the derivative one more time and move the extra item to the same side, we can get $\theta 
_0$ as follows: 
\begin{align*}
    \theta_0 = \phi'''(t)=\frac{1}{\pi i}\partial_t\left(   \frac{\partial_t \left(\frac{\partial_t T^{(g)}_f }{T^{(g)}_f }\right) }{1+\partial_t\left(\frac{T^{(tg)}_f}{T^{(g)}_f}\right)}  \right)\left[2+\partial_t\left(\frac{2\frac{T^{(tg)}_f}{T^{(g)}_f}+\partial_t\left(\frac{T^{(t^2g)}_f}{T^{(g)}_f}\right)}{1+\partial_t\left(\frac{T^{(tg)}_f}{T^{(g)}_f}\right)}\right)\right]^{-1}.
\end{align*}
The rest terms are
\begin{align*}
    \mu_0+\theta_0 t= \phi''(t) =\frac{\partial_t \left(\frac{\partial_t T^{(g)}_f }{T^{(g)}_f }\right)-\pi i \theta_0\left[2\frac{T^{(tg)}_f}{T^{(g)}_f}+\partial_t\left(\frac{T^{(t^2g)}_f}{T^{(g)}_f}\right)\right]}{2\pi i\left(1+\partial_t\left(\frac{T^{(tg)}_f}{T^{(g)}_f}\right)\right)}
\end{align*}
and
\begin{align*}
    \omega_0 + \mu_0 t + \frac{\theta_0}{2} t^2 = \phi'(t)=&\frac{1}{2\pi i}\left(\frac{\partial_t T^{(g)}_f }{T^{(g)}_f }\right)-(\mu_0+\theta_0 t)\left(\frac{T^{(tg)}_f}{T^{(g)}_f}\right)\\
    &\quad-\frac{1}{2}\theta_0\left(\frac{T^{(t^2g)}_f}{T^{(g)}_f}\right).
\end{align*}
To simplify the symbols, we observe that $\phi(x)$ is merely an analytic function and can be expressed as $\phi(x) = \sum_{n=0}^{\infty} \phi^{(n)}(t)\frac{(x-t)^n}{n!}$. Therefore, after some calculation, 
\begin{align}
    \partial_t \left(\frac{\partial_t T^g_f}{T^g_f}\right) = \sum_{j=0}^{\infty}2\pi i \phi^{(j+2)}(t)\frac{1}{(j+1)!}\left[(j+1)\frac{T^{t^jg}_f}{T^{g}_f}+\partial_t\left(\frac{T^{t^{j+1}g}_f}{T^{g}_f}\right)\right].
\end{align}
\begin{definition}[The $j^{th}$ q-operator]\label{q-th_operator}
    Given a representation $T_f^{(g)}$ with the window function $g(x)\in \mathcal{S}(\mathbb{R})$ and the target signal $f\in L^2(\mathbb{R})$, define the $j^{th}$ q-operator by 
    $$
    q^{g}_{f,j}(t,\xi,\lambda)= q^{j}_f(t, \xi, \lambda):= (j+1)\frac{T^{t^jg}_f}{T^{g}_f}+\partial_t\left(\frac{T^{t^{j+1}g}_f}{T^{g}_f}\right).
    $$
\end{definition}
With Definition \ref{q-th_operator}, for $j = 0,1$, one can rewrite the reassignment ingredients as
$$
\theta_0=\dfrac{\dfrac{1}{\pi i}\partial_t\left(\dfrac{\partial_t\left(\frac{\partial_t T^g_f}{T^g_f}\right)}{q^0_f}\right)}{2+\partial_t\left(\dfrac{q^1_f}{q^0_f}\right)}, $$
$$
\mu_0+\theta_0 t=\dfrac{1}{2\pi i }\frac{\left(2q^0_f+\partial_t q^1_f\right)\partial_t \left(\dfrac{\partial_t T^g_f}{T^g_f}\right)-\partial^2_{tt}\left(\dfrac{\partial_t T^g_f}{T^g_f}\right)q^1_f}{2\left(q^0_f\right)^2+\partial_t q^1_f q^0_f-q^1_f\partial_t q^0_f}.
$$
Take $f \in L^2(\mathbb{R})$ and $g \in S(\mathbb{R})$ as the window. For a representation $T_f^{(g)}$, define
\begin{equation}
    \theta_f^{(g)}(t, \xi,\lambda)=\dfrac{\dfrac{1}{\pi i}\partial_t\left(\dfrac{\partial_t\left(\frac{\partial_t T^g_f}{T^g_f}\right)}{q^0_f}\right)}{2+\partial_t\left(\dfrac{q^1_f}{q^0_f}\right)},
\end{equation}
\begin{equation}
\label{eqn:new_mu}
    \begin{aligned}
    \tilde{\mu}^{(g)}_f(t, \xi, \lambda):=\dfrac{1}{2\pi i }\frac{\left(2q^0_f+\partial_t q^1_f\right)\partial_t \left(\dfrac{\partial_t T^g_f}{T^g_f}\right)-\partial^2_{tt}\left(\dfrac{\partial_t T^g_f}{T^g_f}\right)q^1_f}{2\left(q^0_f\right)^2+\partial_t q^1_f q^0_f-q^1_f\partial_t q^0_f},
\end{aligned}
\end{equation}
\begin{equation}
\label{eqn:new_omega}
\begin{aligned}
    \tilde{\omega}^{(g)}_f(t, \xi, \lambda):&=\frac{1}{2\pi i}\left(\frac{\partial_t T^{(g)}_f }{T^{(g)}_f }\right)-\tilde{\mu}^{(g)}_f(t, \xi, \lambda)\left(\frac{T^{(tg)}_f}{T^{(g)}_f}\right)\\
    &\quad-\frac{1}{2}\theta^{(g)}_f(t, \xi, \lambda)\left(\frac{T^{(t^2g)}_f}{T^{(g)}_f}\right).
\end{aligned}
\end{equation}
\begin{definition}[Proposed Higher-Order SCT]
Take $f \in L^2(\mathbb{R})$ as the input function and take $g \in S(\mathbb{R})$
as the window. For a small $\alpha > 0$, define the higher order SCT with resolution $\alpha$ by
    \begin{align*}
    &S_{f}^{(g, \alpha)}(t, \xi, \lambda):=
    \iint_{\mathbb{R}^{2}} T_{f}^{(g)}(t, \eta, \gamma) h_{\alpha}(\xi-\tilde{\omega}_{f}^{(g)}(t, \eta, \gamma)) h_{\alpha}(\lambda-\tilde{\mu}_{f}^{(g)}(t, \eta, \gamma)) \mathrm{d} \eta \mathrm{d} \gamma
    \end{align*}
where $\tilde{\mu}_{f}^{(g)}(t, \xi, \lambda)$, $\tilde{\omega}_{f}^{(g)}(t, \xi, \lambda)$ are define as in \eqref{eqn:new_mu} and \eqref{eqn:new_omega}, respectively. Similarly, $h_{\alpha}$ is an “approximate $\delta$-function".
\end{definition}
\begin{proposition}[Reassignment ingredients for proposed high-order SCT]\label{main_prop.}
Suppose that $T^{(g)}_f$ is the chirplet transform defined in Definition \ref{CT}. Define 
$$ \theta^{(g)}_f(t, \xi, \lambda):=
\dfrac{\dfrac{1}{\pi i}\partial_t\left(\dfrac{\partial_t\left(\frac{\partial_t T^g_f}{T^g_f}\right)}{q^0_f}\right)}{2+\partial_t\left(\dfrac{q^1_f}{q^0_f}\right)},
$$ 
$$\tilde{\mu}^{(g)}_f(t, \xi, \lambda):=\dfrac{1}{2\pi i }\frac{\left(2q^0_f+\partial_t q^1_f\right)\partial_t \left(\dfrac{\partial_t T^g_f}{T^g_f}\right)-\partial^2_{tt}\left(\dfrac{\partial_t T^g_f}{T^g_f}\right)q^1_f}{2\left(q^0_f\right)^2+\partial_t q^1_f q^0_f-q^1_f\partial_t q^0_f},$$
$$\tilde{\omega}^{(g)}_f(t, \xi, \lambda):=\frac{\partial_t T^{(g)}_f }{2\pi iT^{(g)}_f }-\tilde{\mu}^{(g)}_f(t, \xi, \lambda)\left(\frac{T^{(tg)}_f}{T^{(g)}_f}\right)-\frac{1}{2}\theta^{(g)}_f(t, \xi, \lambda)\left(\frac{T^{(t^2g)}_f}{T^{(g)}_f}\right)$$
if $f(x)$ is of the form $A(x)e^{2\pi i\phi(x)}$ , $\log A(x)=\sum_{j=0}^3 [\log A]^{(j)}(t)\frac{(x-t)^j}{j!} $ , and $ \phi(x)$ is a cubic polynomial, then we have
$$Re(\theta^{(g)}_f(t, \xi, \lambda))=\phi'''(t), Re(\tilde{\mu}^{(g)}_f(t, \xi, \lambda))=\phi''(t),$$  $$Re(\tilde{\omega}^{(g)}_f(t, \xi, \lambda))=\phi'(t).$$ 
\end{proposition}
\begin{remark}
    It is obvious that, compared with the SCT in \cite{chen2023disentangling}, the new estimator for $\phi''(t)$ will degenerate to the ingredient given in \cite{chen2023disentangling} when $\theta^{(g)}_f $ vanishes. Therefore, as expected, the result of the proposed method is similar to that of the SCT in the case where the signal is a superposition of pure chirps. 
\end{remark}
This proposition somehow suggests that, if the target signal is close enough to a complex function with a cubic phase, then the new reassignment ingredient will indeed be a suitable estimator of the first three derivatives of the phase function. In the following, we will present some mathematical analysis to explain this fact explicitly. We first consider the following model, which is similar to the $\epsilon$--ICT function in \cite{chen2023disentangling}.
\begin{definition}[The $\epsilon$-intrinsic cubic type function ($\epsilon$-ICBT)]
    Suppose that $\epsilon$ is fixed and larger than $0$. A function $f: \mathbb{R} \rightarrow \mathbb{C}$ is said to be the $\epsilon$-intrinsic cubic type function ($\epsilon$-ICBT) if $f(x) = A(x) e^{i 2 \pi \phi(x)}$ where $A$ and $\phi$ have the following properties:
    $$
    \begin{gathered}
    A \in C^{3}(\mathbb{R}) \cap L^{\infty}(\mathbb{R}), \quad \phi \in C^{4}(\mathbb{R}), \\
    \inf _{x \in \mathbb{R}} \phi^{\prime}(x)>0, \quad \sup _{x \in \mathbb{R}} \phi^{\prime}(x)<\infty, \\
    A(x)>0, \quad\left|A^{\prime}(x)\right|,\left|A^{\prime \prime}(x)\right|,\left|A^{\prime\prime\prime}(x)\right|, \left|\phi^{(4)}(x) \right|\leq \epsilon\left|\phi^{\prime}(t)\right|, \quad \forall x \in \mathbb{R}.
    \end{gathered}
    $$
\end{definition}
The definition specifies the oscillatory function that locally acts like a cubic polynomial phase function with the closeness quantified by $\epsilon$. Furthermore, not only the variation of the IF but also that of the AM is controlled by the IF.
\begin{definition}[Superposition of well-separated $\epsilon$-ICBT components] 
    A function $f$ : $\mathbb{R} \rightarrow \mathbb{C}$ is said to be in the space $\mathcal{Y}_{\epsilon, \Delta}$ of the superposition of well-separated $\epsilon$-ICBT functions if there exists a finite $K$ such that
$$
f(x)=\sum_{k=1}^K f_k(x)=\sum_{k=1}^K A_k(x) e^{2 \pi i \phi_k(x)},
$$
    where each $f_k$ is an $\epsilon$-ICBT function, and their respective phase functions $\phi_k$ satisfy
$$
\left|\phi_k^{\prime}(t)-\phi_l^{\prime}(t)\right|+\left|\phi_k^{\prime \prime}(t)-\phi_k^{\prime \prime}(t)\right| \geq 2 \Delta\,.
$$
\end{definition}
Now, for any $g\in \mathcal{S}(\mathbb{R})$, we define 
$$\mathcal{C}(g(x))(\xi, \lambda):=\int_{\mathbb{R}}g(x)e^{-2\pi i \xi x-\pi i \lambda x^2}dx$$ 
\begin{theorem}[Approximation of the reassignment ingredients]\label{Main_theorem}
Suppose that $f \in \mathcal{Y}_{\epsilon, \Delta}$. Pick a window function $g \in \mathcal{S}(\mathbb{R})$ that satisfies, for all $k = 0,1,2,\cdots, K$ and , $n=0,1,2\cdots$, \\$\left|\mathcal{C}\left(x^{{n}} g(x) e^{i \frac{1}{3}\pi \phi_k'''(t)x^3}\right)(\xi, \eta)\right|\leq \frac{\sqrt{\Delta}D_n\epsilon}{\sqrt{|\xi|+|\eta|}}$ for some $D_n>0$. Let $\tilde{\epsilon}=\epsilon^{\frac{1}{10}}$. Then, provided $\epsilon$ (and thus also $\tilde{\epsilon}$ ) is sufficiently small, the following inequalities hold:

For each tuple $(t, \xi, \lambda) \in M_k$ , where $$
M_k:=\left\{(t, \xi, \lambda):\left|\xi-\phi_k^{\prime}(t)\right|+\left|\lambda-\phi_k^{\prime \prime}(t)\right|<\Delta\right\},
$$such that $\left|T_f^{(g)}(t, \xi, \lambda)\right|>\tilde{\epsilon}$ , $\left|q^0_f(t, \xi, \lambda)\right|>$ $\tilde{\epsilon}$ and $\pi\left|2+\partial_t\left(\dfrac{q^1_f}{q^0_f}\right)\right|>\Tilde{\epsilon}$, we have
$$
\left|\omega_f^{(g)}(t, \xi, \lambda)-\phi_k^{\prime}(t)\right| \leq \tilde{\epsilon}, $$
$$\left|\mu_f^{(g)}(t, \xi, \lambda)-\phi_k^{\prime \prime}(t)\right| \leq \tilde{\epsilon}, $$\text { and }$$|\theta_f^{(g)}(t,\xi,\lambda) - \phi'''_k(t)| \leq \Tilde{\epsilon}. 
$$
\end{theorem}
The proof of this theorem is in the Appendix since it is an extension of Theorem 1 in \cite{chen2023disentangling}. The proof includes some standard estimations as well as some new estimates for $\theta^{g}_f$.
\clearpage
\subsection{Computation for Implementation}
Considering the stability of computation, taking a derivative directly is not a good implementation method. To achieve better implementation results, it is important to substitute all (or nearly all) of the derivatives. Set $T^{(g)}_f$ to be the chirplet transform with window function $g(x)\in \mathcal{S}(\mathbb{R})$ and input signal $f\in \mathcal{L}^2(\mathbb{R})$.
\begin{proposition}[Replacement of the derivative (i)]\label{Elimination}
From Definition \ref{CT}, we have 
\begin{equation}
     \partial_t T_f^{g}=-T_f^{\left(g^{\prime}\right)}+2 \pi i \xi T_f^{g}+2 \pi i \lambda T_f^{t g},
\end{equation}
\begin{equation}
\begin{aligned}
    \partial_{t t}^2 T_f^{g}&=T_f^{g^{\prime \prime}}-4 \pi i \xi T_f^{g^{\prime}}-2 \pi i \lambda T_f^{t g^{\prime}}+(2 \pi i \xi)^2 T_f^{g}\\
    &\hspace{0.5cm}+2(2 \pi i \xi)(2 \pi i \lambda) T_f^{t g}-2 \pi i \lambda T_f^{\left(g+t g^{\prime}\right)}+(2 \pi i \lambda)^2 T_f^{t^2 g},
\end{aligned}
\end{equation}
and 
\begin{equation}
    \begin{aligned}
        \partial_{ttt} T_f^{g} = -T_f^{g'''}&+i6 \pi \xi T_f^{g''}+i6 \pi \lambda T_f^{tg''}\\
    + T_f^{g'} &\left[i6 \pi \lambda-12(i\pi \xi)^2\right]
    + T_f^{tg'} \left[-24 (i \pi \xi)(i  \pi \lambda)\right]\\
    + T_f^{g} &\left[8(i \pi \xi)^3-12(\pi i \lambda)(i \pi \xi)\right]
    + T_f^{t^2g'} \left[-12(i \pi \lambda)^2\right]\\
    + T_f^{tg} &\left[24(i \pi \xi)^2(\pi i \lambda)-12(i \pi \lambda)^2\right]
    + T_f^{t^2g} \left[24(i \pi  \lambda)^2(i \pi \xi)\right]\\
    + T_f^{t^3g} &\left[8(i \pi \lambda)^3\right]\,,
    \end{aligned}
\end{equation}
in which $$T_f^{g'},\,T_f^{tg},\, T_f^{g''},\,T_f^{tg'},\,T_f^{t^2g},\,T_f^{g'''},\,T_f^{tg''},\,T_f^{t^2g'},\,T_f^{t^3g}$$ are respectively the CTs of f with the window functions
\begin{equation*}
\,g'(t),\,tg(t),\,g''(t),\,tg'(t),\,t^2g(t),\,g'''(t),\,tg''(t),\,t^2g'(t),\,t^3g(t).
\end{equation*}
\end{proposition}
\begin{proposition}[Replacement of the derivative (ii)]
To replace the derivative of the denominator in \eqref{eqn:new_mu}, we compute 
\begin{equation}
    \begin{aligned}
    \partial_t \left(\dfrac{\partial_t T^{g}_f}{T^{g}_f}\right)= \dfrac{1}{(T^{g}_f)^2}\left[\partial_{tt} T^{g}_fT^{g}_f-\left(\partial_t T^{g}_f\right)^2\right]
\end{aligned}
\end{equation}
and 
\begin{equation}
    \begin{aligned}
    \partial_{tt}\left(\dfrac{\partial_t T^{g}_f}{T^{g}_f}\right) =\dfrac{1}{(T^{g}_f)^4}\left[\left(\partial_{ttt}T^{g}_f T^{g}_f+\partial_{tt}T^{g}_f\partial_t T^{g}_f-2\partial_tT^{g}_f\partial_{tt}T^{g}_f\right)(T^{g}_f)^2\right.\\
    \left.-2T^{g}_f\partial_tT^{g}_f\left(\partial_{tt} T^{g}_fT^{g}_f-\left(\partial_t T^{g}_f\right)^2\right)\right]\, .
\end{aligned}
\end{equation}

\end{proposition}
\begin{proposition}[Replacement of the derivative for $q^j_f$]
For $j=0,1$, one has
\begin{equation}
    \begin{aligned}
        q^j_f(t, \xi, \lambda)=\frac{1}{\left(T^g_f\right)^2}\left[T^{t^{j+1}g}_fT^{g'}_f-T^{t^{j+1}g'}_fT^{g}_f+i 2\pi  \lambda\left(T^{t^{j+2}g}_fT^{g}_f-T^{t^{j+1}g}_fT^{tg}_f\right)\right],
    \end{aligned}
\end{equation}
\begin{equation}
\begin{aligned}
\partial _t q^j_f(t,\xi, \lambda) &= \frac{1}{(T^g_f)^2}\left[\left(\partial_t T^{t^{j+1}g}_fT^{g'}_f+T^{t^{j+1}g}_f\partial_t T^{g'}_f-\partial_t T^{t^{j+1}g'}_fT^{g}_f-T^{t^{j+1}g'}_f\partial_t T^{g}_f\right)  \right.\\
&\hspace{2cm}+\left.2\pi i \lambda\left(\partial_t T^{t^{j+2}g}_fT^{g}_f+T^{t^{j+2}g}_f\partial_t T^{g}_f-\partial_t T^{t^{j+1}g}_fT^{tg}_f\right.\right.\\
&\hspace{3.5cm}\left.\left.-T^{t^{j+1}g}_f\partial_t T^{tg}_f\right)\right]\\
&\quad-\frac{2\partial_t T^g_f}{(T^g_f)^3}\left[T^{t^{j+1}g}_fT^{g'}_f-T^{t^{j+1}g'}_fT^{g}_f+2\pi i \lambda\left(T^{t^{j+2}g}_fT^{g}_f-T^{t^{j+1}g}_fT^{tg}_f\right)\right],
\end{aligned}
\end{equation}
where $a \in \mathbb{N}$ and $b\in \mathbb{N}\cup\{0\}$,
\begin{equation}
\begin{aligned}
    \partial_t T^{t^ag^{(b)}}_f(t,\xi, \lambda) = -aT^{t^{a-1}g^{(b)}}_f-T^{t^ag^{(b+1)}}_f+2\pi i \xi T^{t^ag^{(b)}}_f+2\pi i \lambda T^{t^{a+1}g^{(b)}}_f 
\end{aligned}
\end{equation}
\end{proposition}

\subsection{Proposed Algorithm}
\begin{spacing}{1.1}
\begin{algorithm}[H]
\caption{Higher Order Synchrosqueezed Chirplet Transforms (HSCT)}
\begin{flushleft}
\hspace*{\algorithmicindent} \textbf{Input} Signal $\mathbf{x}\in\mathbb{C}^n$ or $\mathbb{R}^n$, sampling rate $f_s$, window vector $\mathbf{g}\in\mathbb{C}^{2Q+1}$ or $\mathbb{R}^{2Q+1}$ and suitable $N\geq 2Q+1$\\
\hspace*{\algorithmicindent} \textbf{Output} $\mathbf{V}^{(g)}_x\in\mathbb{C}^{n\times F\times C}$ where $F, C$ are the length of frequency axis and chirp axis, respectively.\\
\end{flushleft}
\label{alg:lead}
\begin{algorithmic}[1]
\State{Build the time ticks $\mathbf{t}\in\mathbf{R}^n$}
\State{Compute resolution of frequency $df=\frac{f_s}{N}=\frac{1}{dt\cdot N}$} and chirp $d\lambda=\frac{2f_s^2}{N^2}$
\State{Build the input frequency ticks $\mathbf{f}\in\mathbb{R}^{N}$ and chirp rate ticks $\mathbf{c}\in\mathbb{R}^{N}$}
\State{Compute the derivative of window function, $\mathbf{g}'$, $\mathbf{g}''$, $\mathbf{g}'''$}
\State Let $F=\frac{\mathbf{f}[-1]}{\alpha}$ and $C=\frac{2\mathbf{c}[-1]}{\alpha}$
\State Create a $n\times F\times C$ matrix $\mathbf{V}^{(g)}_x$.
\For{$i=1\text{~to~}n$}
\For{$k=1\text{~to~}C$}
\State Let $\mathbf{y}_{a,b}=\mathbf{x}[i-Q:i+Q]\cdot (\mathbf{t}[i-Q:i+Q]-\mathbf{t}_i)^b\cdot \mathbf{g}^{(a)}\cdot$\hspace{3cm}
\linebreak\hspace*{\algorithmicindent}\hfill$\exp\left(-i\pi\mathbf{c}_k(\mathbf{t}[i-Q:i+Q]-\mathbf{t}_i)^2\right)$
\State Compute $\mathbf{Y}^{(x^bg^{(a)})}=dt\cdot\exp(i2\pi\frac{0:(N-1)}{N})\cdot$\func{fft}$(\mathbf{y}_{a,b}$, $N)\in\mathbb{C}^N$ 
\State {Choose the subvector of $\mathbf{Y}^{(x^bg^{(a)})}$ corresponding to non-negative\quad\quad\quad\quad \linebreak\hspace*{\algorithmicindent}frequency. \Comment{Approximate $T^{(x^bg^{(a)})}_x(t_i,:,\lambda_k)$}}
\State Rewrite $\mathbf{Y}^{(x^bg^{(a)})}=\mathbf{Y}^{(x^bg^{(a)})}$
\State Compute $\partial_t\mathbf{Y}^{(x^bg^{(a)})}$, $\partial_{tt}\mathbf{Y}^{(x^bg^{(a)})}$ and $\partial_{ttt}\mathbf{Y}^{(x^bg^{(a)})}$
\State Compute $\mathbf{q}_0$, $\mathbf{q}_1$, $\partial_t\mathbf{q}_0$, $\partial_t\mathbf{q}_1$, $\partial_t\left(\frac{\partial_t \mathbf{Y}^{(x^bg^{(a)})}}{\mathbf{Y}^{(x^bg^{(a)})}}\right)$ and $\partial_{tt}\left(\frac{\partial_t \mathbf{Y}^{(x^bg^{(a)})}}{\mathbf{Y}^{(x^bg^{(a)})}}\right)$
\State Compute $\boldsymbol{\theta}$, $\boldsymbol{\mu}$, $\boldsymbol{\omega}$
\State Normalized $\boldsymbol{\mu}$, $\boldsymbol{\omega}$ to $F$, and $C$ as \var{muIdx} and \var{omegaId}
\For{$j=1\text{~to~}F$}
\If{\func{abs}$(\mathbf{Y}^{(g)}[j])$>\var{thres}} \Comment{Avoid denominator close to zero}.
\State Reassign $\mathbf{V}^{(g)}_{\mathbf{x}}[i, \var{omegaIdx}[j], \var{muIdx}[j]] += \mathbf{Y}^{(x^bg^{(a)})}[j]$
\EndIf
\EndFor
\EndFor
\EndFor
\end{algorithmic}
\end{algorithm}
\end{spacing}

\clearpage

\section{Experimental Results}
\label{sec:results}
In this section, we present some experimental results to show the improvement given by the proposed method, utilizing the window function $g_0(t) = e^{-\pi t^2}$. We evaluate the effectiveness of each method from various viewpoints, quantifying the energy concentration of the distribution using the Rényi entropy and assessing the fidelity using $2$- and $3$-dimensional Earth Mover's Distances (EMDs), as detailed in the appendix. In the discussion on the EMD, two scenarios are considered. The first one is to project the $3$-dimensional TFC representations onto the $2$-dimensional TF representation, defined as follows:
\begin{equation*}
\mathfrak{T}f^{(g)}(t, \xi):=\int_{-\infty}^{\infty}\left|T_f^{(g)}(t, \xi, \lambda)\right| \mathrm{d} \lambda.
\end{equation*}
Here, by fixing $t$, we compute the one-dimensional EMD of the frequency distribution, followed by averaging the EMD curves over all time instances. The second scenario is to directly address the $3$-dimensional TFC representation. By fixing the time $t$, we compute the $2$-dimensional EMD of frequency and chirp distributions and then average the EMD over all time instances.
\subsection{The Results for Large Chirp-Rate Variation}
Recall that the most suitable condition for the SCT is that the target signal is a pure chirp signal or close to a pure chirp signal, which is called the $\epsilon$--intrinsic chirp type function ($\epsilon$--ICT function) \cite{chen2023disentangling}. However, when this is not the case, then the performance of the SCT will be limited.

The large variation in the instantaneous frequency will make the SCT perform badly. By contrast, the proposed method can linearly detect the variance in the chirp rate and hence has much better performance. Consider the signal with a varied chirp rate as follows  
    $$
   \begin{aligned}
    & x_{2}=\exp \left(i 2 \pi\left(12t + 12\exp(-\frac{1}{2}(t-3)^2)\right)\right)\,.
\end{aligned}
   $$
Note that $x_2$ is a signal with a heat kernel in its phase term and has a varied chirp rate. The instantaneous frequency for $x_{2}$ is $ f_{2} = 12- 12(t-3)\exp(-\frac{1}{2}(t-3)^2)$.
\begin{table}[h]
\caption{Evaluation for the signal $x_2$ using different metrics.  Comparison between our method and the previous method.}
\begin{center}
\begin{tabular}
{
>{\centering\arraybackslash}m{2.2cm} ||
>{\centering\arraybackslash}m{2.9cm}
>{\centering\arraybackslash}m{2.6cm}
>{\centering\arraybackslash}m{2.6cm}
>{\centering\arraybackslash}m{2.6cm}
}
\noalign{\hrule height 1.5pt}
Method  & Rényi Entropy & TF EMD & TFC EMD  \\
\noalign{\hrule height 1.5pt}
CT  & 22.6166 & 3.5035 &12.0833\\
SCT  & 15.6369 & 0.5583 &2.2828\\
Proposed  & 13.8553 & 0.4017 &1.6203 \\
\noalign{\hrule height 1.5pt}
\end{tabular}
\end{center}
\label{tab:heat}
\end{table}

According to Fig.\ref{scatter}, the scattering plot of the EMD or the magnitude of $|\phi'''(t)|$ can tell us some stories. From the left subfigure in Fig. \ref{scatter}, we can see that $|\phi'''(t)|$ is far from $0$ and the chirp rate of the signal varies very fast. In the middle subfigure, the EMD is plotted versus the time. Considering the magnitude of $|\phi'''(t)|$ at the time instant and plotting the scattering, the right subfigure is obtained. This subfigure can probably be divided into left and right halves. In the right-half part of the subfigure, the distance gap between the two methods becomes wider when the magnitude of the $3^{rd}$ derivative of the phase function gets larger. The left-hand part consists of two kinds of points. One of them is the points near $t = 3$, where $|\phi'''(t)|$ vanishes, and the other one is the complicated boundary points due to the so-called boundary effect. Also, observing Fig. \ref{fig:heat_3D_plotting}, the subfigure in the $3^{rd}$ column of the $2^{nd}$ row shows that the TFC distribution of the SCT method almost collapses as the chirp rate varies rapidly. By contrast, the result in the lower-right subfigure shows that, when using the proposed HSCT method, the obtained TFC distribution is well connected everywhere.
\begin{figure}[ht]
    \centering
    \includegraphics[trim = 100 0 100 0, width=.95\linewidth]{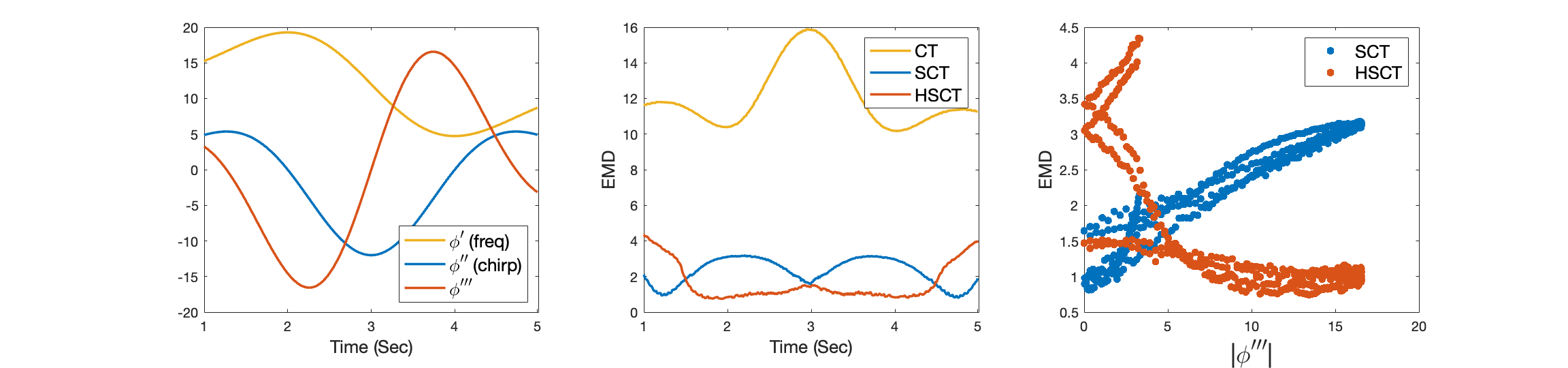}
    \caption{(Left) The numerical value of the derivatives of the phase function from the first to the third orders. (Middle) The EMD of the results given by the CT, the SCT, and the proposed method. (Right) The scattering plot of $|\phi'''(t)|$--EMD.}
    \label{scatter}
\end{figure}
\begin{figure}[h]
\centering
\includegraphics[trim = 250 100 250 100, width=1\linewidth]{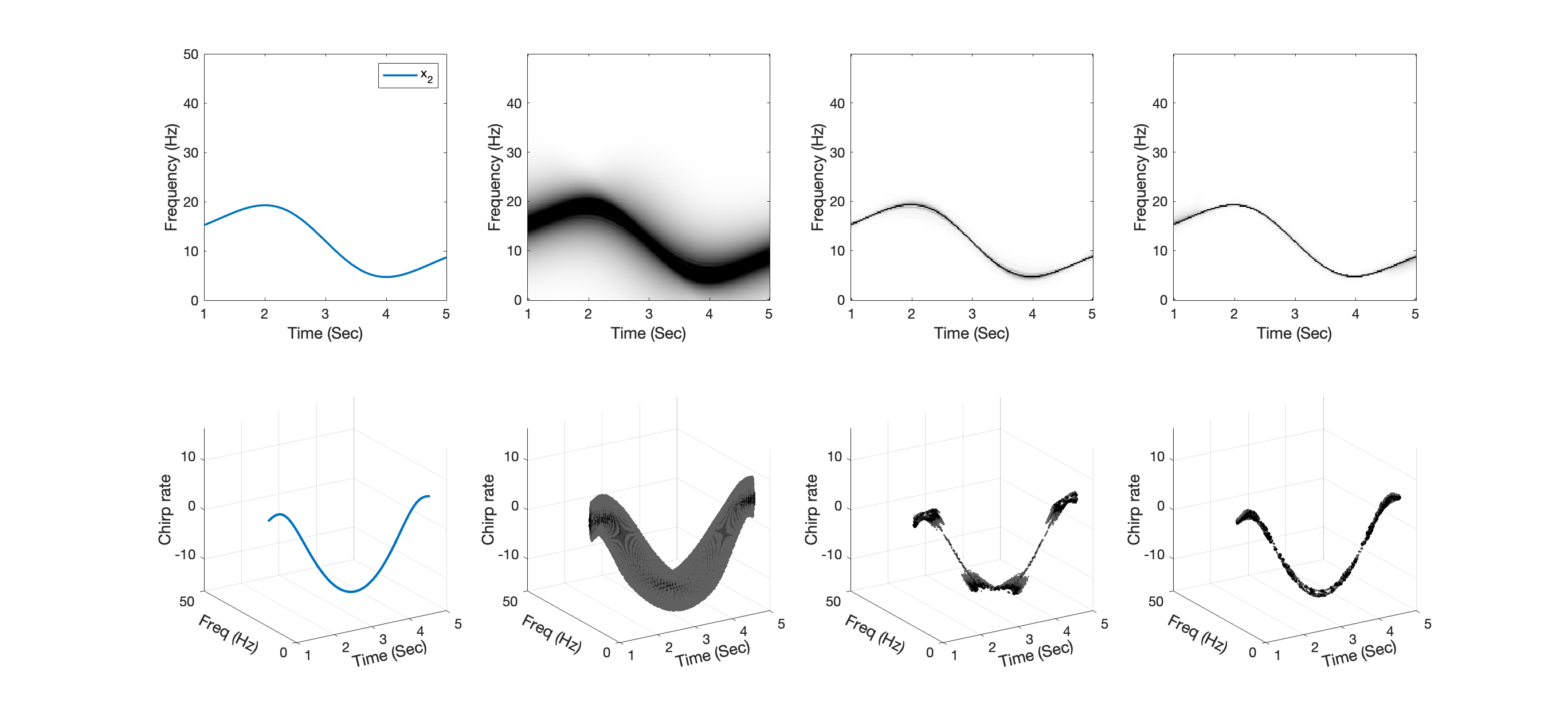}
\caption{(Top row) The visualization of the time-frequency (TF) representation for $x_2$. From left to right: the ideal TF representation, the result of the CT, the result of the original SCT, and the result of the proposed HSCT. (Bottom row) The $3$-dimensional visualization of the time-frequency-chirp (TFC) representation of $x_2$. From left to right: the ideal TFC representation, the result of the CT, the result of the original SCT, and the result of the proposed method.}
\label{fig:heat_2D_plotting}
\label{fig:heat_3D_plotting}
\end{figure}
All of the above experimental results indicate that the variation of the chirp rate is strongly related to the performance of the SCT but the proposed method can overcome the limitation.
\subsection{Separating Signals}
The main motivation to develop the synchrosqueezing transform based on the chirplet transform in \cite{chen2023disentangling} was to weaken the necessary condition of separating the signals, and it would be an important issue of how the proposed method works when dealing with such problems.

We recall the first example we presented in Section \ref{degenerate}, 
$$\begin{aligned}
        & x_{11}=\exp \left(i 2 \pi\left(4 t^2\right)\right)  \\
        & x_{12}=\exp \left(i 2 \pi\left(-\pi t^2+(24+6 \pi) t\right)\right)  \\
        & x_{1}=x_{11}+x_{12} \,,
\end{aligned}$$
the superposition of a pair of pure chirp signals.
\begin{figure}[ht]
    \includegraphics[trim = 100 0 100 0, width=.95\linewidth]{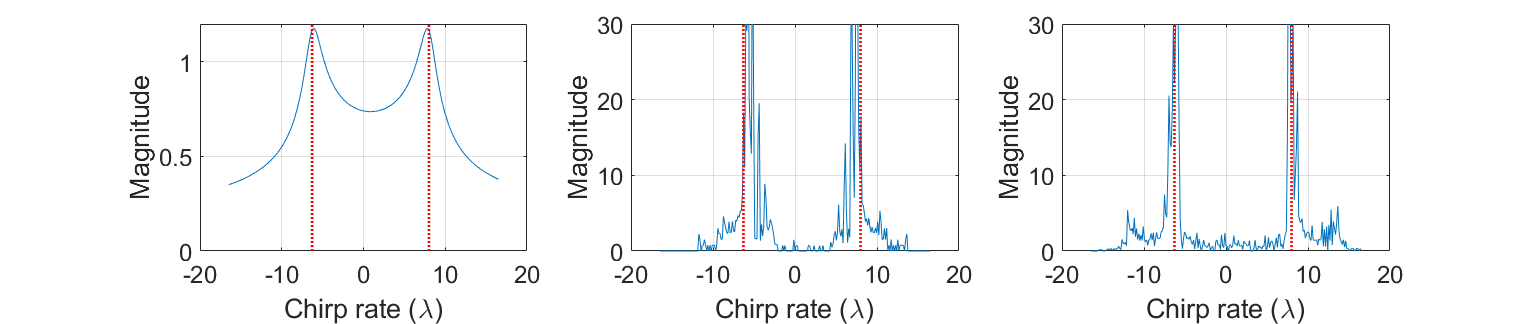}
    \caption{Magnitude of the analyzed chirp rate at $t = 3$ and $\xi = 24$ of $x_1$. The red dotted lines indicate the accurate positions for the impulses. (Left) The result of the CT. (Middle) The result of the SCT. (Right) The result of the proposed method. }
    \label{linear_CR-mag}
\end{figure}
In Fig. \ref{linear_CR-mag}, the chirp-rate analysis result of this synthetic example is shown. One can see that, when using the proposed method, the magnitude of the chirp rate of the two components is well-separated, which shows the proposed method works well in this case.

Next, we explore the difference between the two methods in more complex examples. Consider the synthetic signal defined as the superposition of a chirp function and a signal with a sine curve phase function as follows. 
$$\begin{aligned}
        & x_{31}=\exp \left(i 2 \pi\left(8 (t-2.2)^2\right)\right),  \\
        & x_{32}=\exp \left(i 2 \pi\left(-2 \cos \left(\frac{2}{3} \pi (t-2)\right)+13 (t-2)\right)\right), \\
        & x_{3}=x_{31}+x_{32}\,.
    \end{aligned}$$
The instantaneous frequency for $x_{31}$ and $x_{32}$ are $ f_{31}=16 (t-2.2)$ and $f_{32}=\frac{4}{3} \pi \sin \left(\frac{2}{3} \pi (t-2)\right)+13$, respectively, and the intersection of them is at about $t \approx 3.176$ and $\xi \approx 15.67$.

Some differences arise in this example between the methods of the SCT and the proposed HSCT. As shown in Fig. \ref{sine_CR-mag}, the magnitude of the analyzed chirp rates of the two methods are different. The correct chirp rates of this example at $t = 3.176$ and $\xi = 15.67$ should be two impulses at about $16$ and $-6.83$. The former is contributed from $x_{31}$ and the latter is from $x_{32}$. From Fig. \ref{sine_CR-mag}, the middle subfigure shows the result given by the SCT. The lower order of the SCT leads to the bias of the estimation result, which consequently makes the squeezing process wrong. By comparison, from the right subfigure, one can see that the proposed method gives a quite reliable estimation result of the chirp rate.

\begin{figure}[ht]
\centering
\includegraphics[trim = 100 0 100 0, width=.95\linewidth]{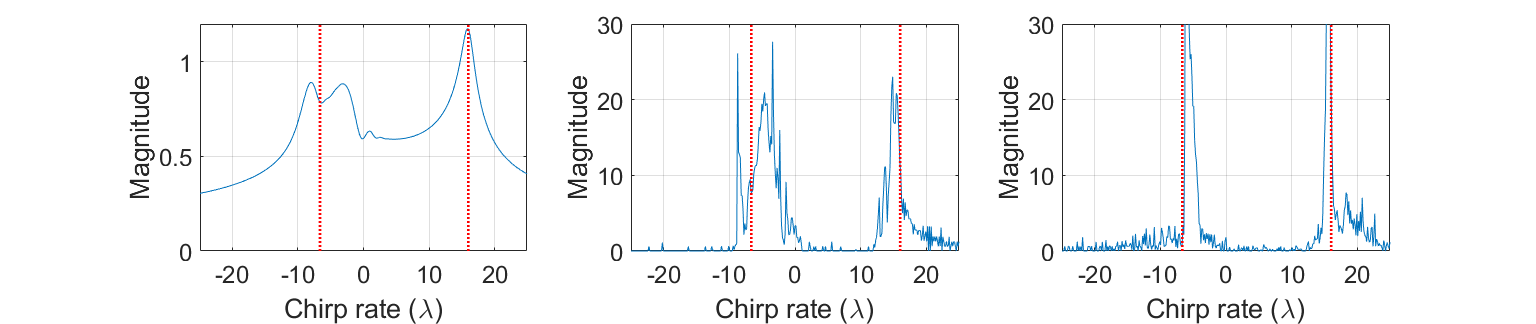}
\caption{Magnitude of the analyzed chirp rate at $t = 3.176$ and $\xi = 15.67$ of $x_3$. The red dotted lines indicate the accurate positions for the impulses. (Left) The result of the CT. (Middle) The result of the SCT. (Right) The result of the proposed HSCT method.}
\label{sine_CR-mag}
\end{figure}
\begin{figure}[ht]
\centering
\includegraphics[trim = 250 100 250 100, width=1\linewidth]{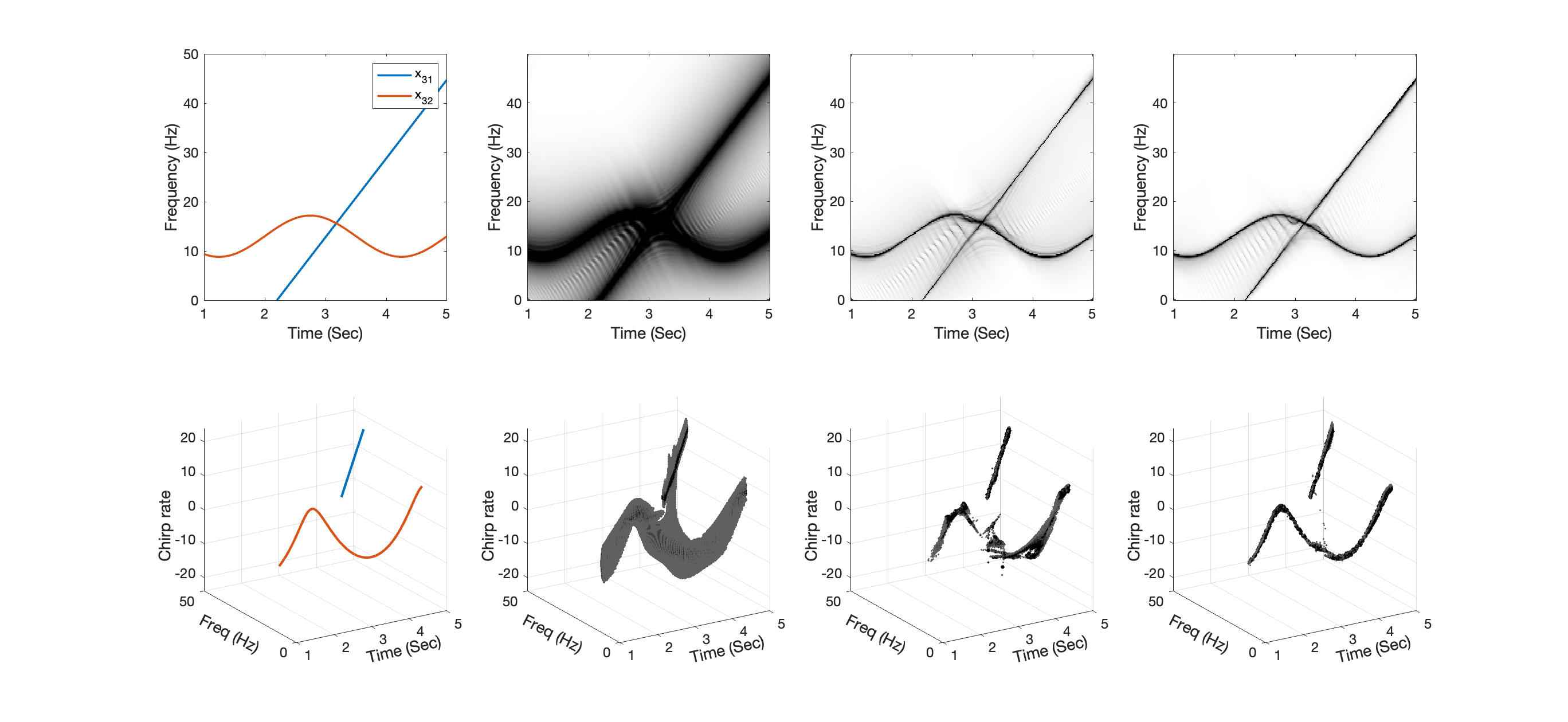}
\caption{(Top row): Visualizations of the TF representations for $x_3$. From left to right: the ideal TF representation, the results of the CT, the original SCT, and the proposed HSCT, respectively. (Bottom row): Three dimensional visualizations of the TFC representations of $x_3$. From left to right: the ideal TFC representation, the results of the CT, the original SCT, and the proposed HSCT, respectively.}
\label{fig:sine_2D_plotting}
\label{fig:sine_3D_plotting}
\end{figure}
From the $3^{rd}$ column in the bottom row of Fig. \ref{fig:sine_2D_plotting}, one can see that the TFC representation determined by the SCT contains more side lobes and the sinusoidal part of the TFC representation is blurred. One can imagine that the blurring of the sinusoidal part is due to the fast variation of the chirp rate, which reduces the applicability of the SCT. In contrast, from the $4^{th}$ column of Fig. \ref{fig:sine_2D_plotting}, one can see that the proposed method not only gives a clear string wave graph but also well separates the two modes in the TFC space. According to Table \ref{tab:sine}, the proposed HSCT method has a lower Rényi entropy and lower TF and TFC EMDs, which means that the proposed method has better energy concentration.
\begin{table}[H]
\caption{Evaluation for the signal $x_3$ case using different metrics.  Comparison between our method and the previous method.}
\begin{center}
\begin{tabular}
{
>{\centering\arraybackslash}m{2.2cm} ||
>{\centering\arraybackslash}m{2.9cm}
>{\centering\arraybackslash}m{2.6cm}
>{\centering\arraybackslash}m{2.6cm}
>{\centering\arraybackslash}m{2.6cm}
}
\noalign{\hrule height 1.5pt}
Method  & Rényi Entropy & TF EMD & TFC EMD  \\
\noalign{\hrule height 1.5pt}
CT  &  24.1379 & 4.4416 &14.3708\\
SCT  & 18.4175 & 2.9218 &11.1113\\
Proposed  & 17.7811 & 1.7455 &7.5175 \\
\noalign{\hrule height 1.5pt}
\end{tabular}
\end{center}
\label{tab:sine}
\end{table}
\subsection{Degenerated into the SCT.}\label{degenerate}
As mentioned in Section 3, the proposed method adds the parameter $\theta^{(g)}_f$ to adjust the estimation result of the phase function. Its purpose is to include the information of the $3^{rd}$ derivative of the phase function in the synchrosqueezing transform. However, when the target signal is a pure chirp or the composition of pure chirps, the extra adjustment becomes a redundancy. 

We now consider the case where the signal is a composition of a pair of pure chirps. The input signal is as follows 
$$\begin{aligned}
        & x_{11}=\exp \left(i 2 \pi\left(4 t^2\right)\right)  \\
        & x_{12}=\exp \left(i 2 \pi\left(-\pi t^2+(24+6 \pi) t\right)\right)  \\
        & x_1=x_{11}+x_{12}
\end{aligned}$$
where the instantaneous frequency for $x_{11}$ and $x_{12}$ are $ f_{11}=8 t$ and $f_{12}=-2 \pi t+(24+6 \pi)$, respectively.

\begin{table}[h]
\caption{Evaluation for the signal $x_1$ using different metrics.  Comparison between our method and the previous method.}
\begin{center}
\begin{tabular}
{
>{\centering\arraybackslash}m{2.2cm} ||
>{\centering\arraybackslash}m{2.9cm}
>{\centering\arraybackslash}m{2.6cm}
>{\centering\arraybackslash}m{2.6cm}
>{\centering\arraybackslash}m{2.6cm}
}
\noalign{\hrule height 1.5pt}
Method  & Rényi Entropy & TF EMD & TFC EMD  \\
\noalign{\hrule height 1.5pt}
CT  & 23.5241 & 3.3526 & 10.1349 \\
SCT  & 15.0726 & 1.1402 & 3.3042 \\
proposed  & 16.0366 & 0.7263 & 3.1288\\
\noalign{\hrule height 1.5pt}
\end{tabular}
\end{center}
\label{tab:quad}
\end{table}
Derived from Section 3, $\theta^{(g)}_f(t,\xi,\lambda)$ represents the magnitude of the $3^{rd}$ derivative of the phase function at time $t$ and frequency $\xi$. 
This leads to the conclusion that the impact of $\theta^{(g)}_f$ is minimal in this case, aligning well with the characteristics of the given signals.  Also, as seen from Table \ref{tab:quad}, the SCT method outperforms other methods in energy concentration based on the winning Rényi entropy, despite similar scores in the two kinds of the EMD. Their TF and TFC representations are shown in Fig. \ref{fig:linear_2D_plotting} in the appendix.
\subsection{More than Two Modes.}
For the last example, we explore the performance when the input is a superposition of more than two modes. The following three-mode signal $x_4$ consists of one linear chirp and two signals with cubic polynomial phase functions.
    $$
\begin{aligned}
    & x_{41}=\exp \left(j 2 \pi\left(\frac{5}{6}t^3-\frac{15}{2}t^2+40t+180\right)\right),  \\
    &x_{42} = \exp \left(j 2 \pi\left(-\frac{5}{6}t^3+\frac{15}{2}t^2+\frac{35}{2}t+\frac{5}{6}\right)\right),\\
    &x_{43} = \exp \left(j 2 \pi\left(-\frac{7}{2}t^2+(26+6 \pi) t-6\pi-\frac{45}{2}\right)\right),  \\
    &x_{4} = x_{41}+x_{42}+x_{43}\,,
\end{aligned}
   $$
where the instantaneous frequency for $x_{41}$, $x_{42}$, and $x_{43}$ are $ f_{41} = \frac{5}{2}t^2-15t+40$, $f_{42}=-\frac{5}{2}t^2+15t+\frac{35}{2}$, and $f_{43}=-7t+26+6\pi$, respectively. The numerical results in Table \ref{tab:three} show that the proposed HSCT also works well when the number of modes are more than two. The visualization result can be seen from Fig. \ref{fig:three_2D_plotting} in the appendix.
   
   \begin{table}[H]
\caption{Evaluation for the signal $x_4$ using different metrics.  Comparison between our method and the previous method.}
\begin{center}
\begin{tabular}
{
>{\centering\arraybackslash}m{2.2cm} ||
>{\centering\arraybackslash}m{2.9cm}
>{\centering\arraybackslash}m{2.6cm}
>{\centering\arraybackslash}m{2.6cm}
>{\centering\arraybackslash}m{2.6cm}
}
\noalign{\hrule height 1.5pt}
Method  & Rényi Entropy & TF EMD & TFC EMD  \\
\noalign{\hrule height 1.5pt}
CT  & 23.7805& 2.8182 & 8.5212\\
SCT  & 20.3239 & 1.9534 & 6.7796\\
Proposed  & 18.1902 & 1.3703 & 5.1615\\
\noalign{\hrule height 1.5pt}
\end{tabular}
\end{center}
\label{tab:three}
\end{table}

\section{Conclusion}
In this paper, the HSCT, which is to extend the synchrosqueezed chirplet transform into a higher-order version, was proposed. It enables us to well analyze the signals with the fast variation of the chirp rate. Experimental results also show that the proposed HSCT has better performance than existing methods. In future work, we will further extend the proposed idea. Note that the proposed new operator is independent of the kernel of the original integral operator. For example, the proposed operator is also valid for the synchrosqueezed wavelet transform and we can apply a similar method to extend the synchrosqueezed wavelet transform into the higher-order version.
\clearpage

\section*{Acknowledgments}
We would like to thank Prof. Hau-Tieng Wu for suggesting this topic and fruitful discussion.

\appendix
\section{Metrics for Implement}

\subsection{Rényi Entropy}
To evaluate a time-frequency representation(TFR), it is important to see if the “energy" is sufficiently concentrated. That is, we expect the model to get closer to the ideal TFR. Based on \cite{stankovic2001measure}, we introduce some metrics to measure the concentration.
\label{sec:sample:appendix}
We adopt the Rényi entropy as the metric. Given a positive real number $\alpha \neq 1$ and a random variable $X$, defined on a probability space $(\Omega, \Sigma, \mathbb{P})$ and valued on a finite state space $\mathcal{S} = \{x_{1}, x_{2}, \cdots ,x_n\}$, one can define its Rényi entropy of order $\alpha $  by 
\begin{equation}
H_{\alpha}(X) = \dfrac{1}{1-\alpha}\log \left(\sum_{i=1}^{n}p_i^{\alpha}\right)
\end{equation}

where the $p_i = \mathbb{P}(X= x_i)$ for $i = 1,2, \cdots ,n$.

This index is applied to measure the concentration of a distribution, which has a lower bound $0$ when the distribution of $X$ is a unit impulse and reaches its upper bound $\log n$ when $X$ has a uniform distribution. We consider the normalized amplitude of TFR to be a probability mass function. Then
$$
p_i = \mathbb{P}(X= x_i) \sim \dfrac{|TFR_f(t_0, \xi_0)|}{\left(\sum_t\sum_\xi |TFR_f(t, \xi)|\right)}
$$
\begin{equation}
    \begin{aligned}
        H_{\alpha}(X)\sim(1-\alpha)^{-1}\log \sum_t\sum_\xi\left(\dfrac{|TFR_f(t_0, \xi_0)|}{\left(\sum_t\sum_\xi |TFR_f(t, \xi)|\right)}\right)^{\alpha}
    \end{aligned}
\end{equation}

which is called the Rényi Entropy of the distribution $TFR_f$.
\subsection{Earth mover's distance}
The Earth Mover's Distance (EMD) provided by \cite{rubner2000} is a cross-bin distance that addresses the image or histogram alignment problem. The common approach is determined through the Earth mover's distance (EMD), a technique previously employed in the context of synchrosqueezing as documented in \cite{daubechies2016} and \cite{pham2017high}. Furthermore, EMD was improved to Fast EMD with thresholded ground distances in \cite{pele2008} and \cite{pele2009}. As ``synchrosqueezing'' operates on the principle of ``reassigning'' content within the time-frequency plane while keeping the time variable fixed, we calculate the EMD for each specific time variable and then take the average over all $t$. It is noteworthy that in the three-dimensional time-frequency-chirp representation, as we fix a specific time, there are still two variables, namely frequency and chirp. Therefore, the units for calculating EMD are in real frequency, not normalized frequency.
\section{Figures}
\begin{figure}[H]
\centering
\includegraphics[trim = 250 100 250 100, width=1\linewidth]{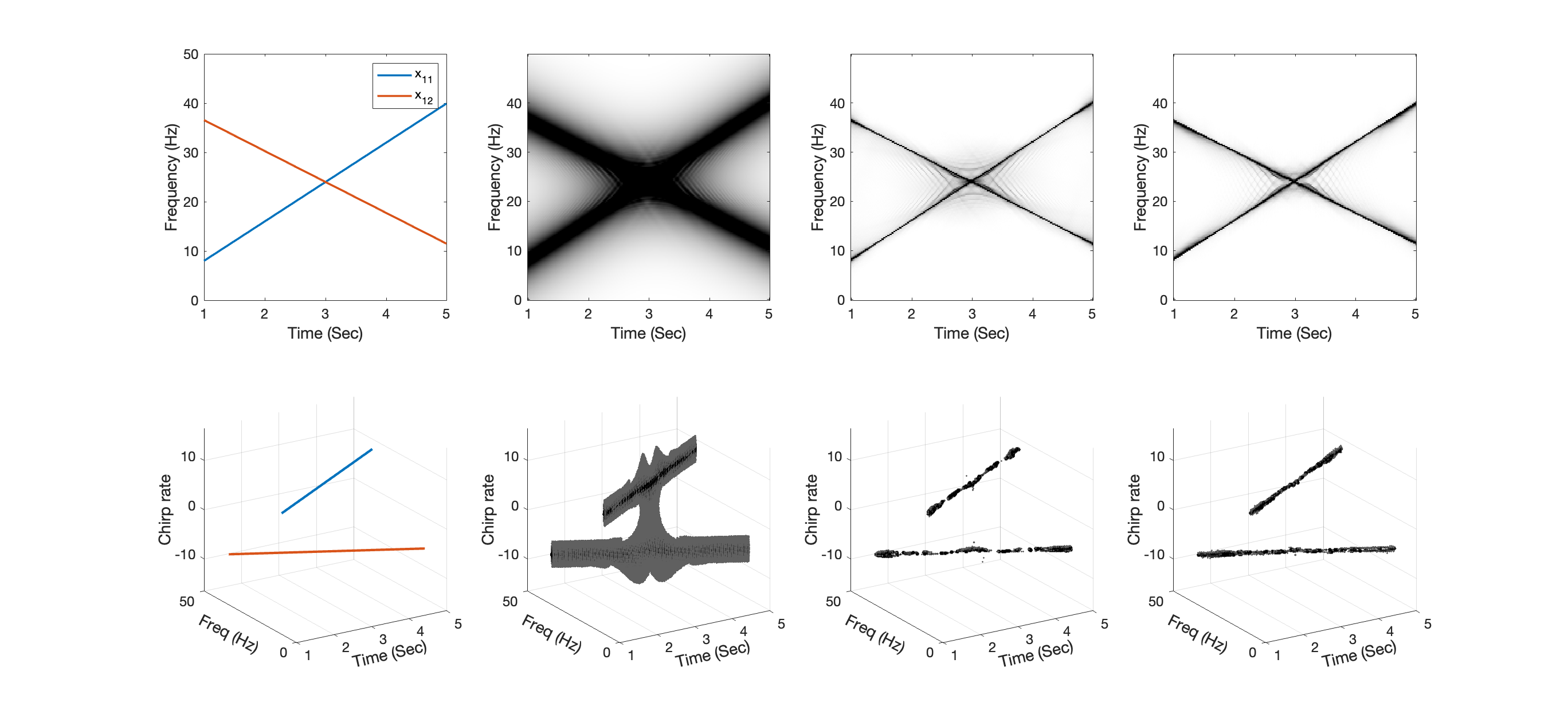}
\caption{(Top row): Visualizations of the TF representations for $x_1$. From left to right: the ideal TF representation, the results of the CT, the original SCT, and the proposed HSCT, respectively. (Bottom row): Three-dimensional visualizations of the TFC representations of $x_1$. From left to right: the ideal TFC representation, the results of the CT, the original SCT, and the proposed HSCT, respectively.}
\label{fig:linear_2D_plotting}
\label{fig:linear_3D_plotting}
\end{figure}
\begin{figure}[H]
\centering
\includegraphics[trim = 250 100 250 100, width=1\linewidth]{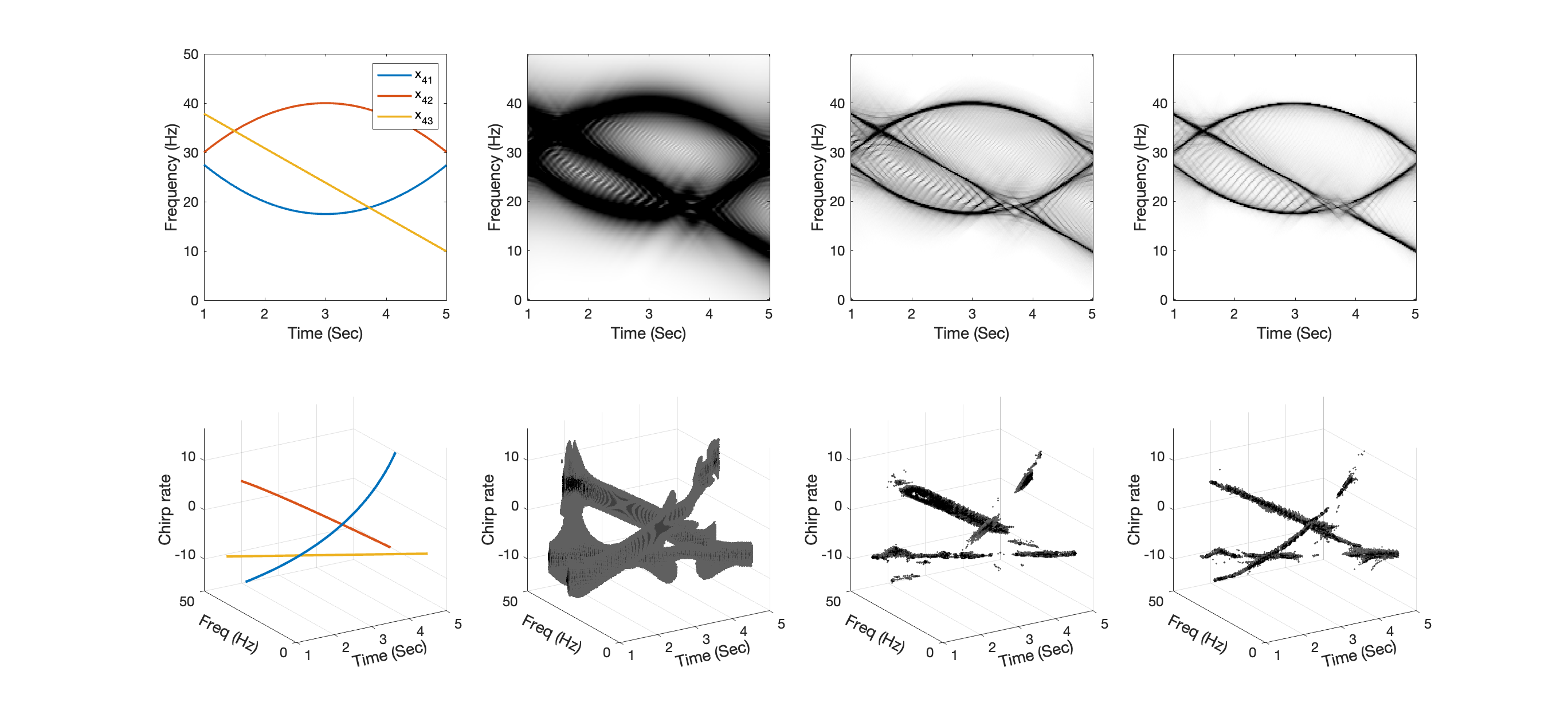}
\caption{(Top row): Visualizations of the TF representations for $x_4$. From left to right: the ideal TF representation, the results of the CT, the original SCT, and the proposed HSCT, respectively. (Bottom row): Three-dimensional visualizations of the TFC representations of $x_4$. From left to right: the ideal TFC representation, the results of the CT, the original SCT, and the proposed HSCT, respectively.}
\label{fig:three_2D_plotting}
\label{fig:three_3D_plotting}
\end{figure}

\section{Proof of Theorem \ref{Main_theorem}}
In this section, we provide the proof of Theorem \ref{Main_theorem} by listing a few lemmas. These lemmas make up the proof of Theorem \ref{Main_theorem}. We will skip the proof if the content is highly similar to the one in the reference.
\begin{lemma}\label{Lemma0}
For any tuple $(t, \xi, \lambda)$ under consideration, there can be at most one $k \in\{1, \ldots, K\}$ such that $\left|\xi-\phi_k^{\prime}(t)\right|+\left|\lambda-\phi_k^{\prime \prime}(t)\right|<\Delta$, i.e. $(t, \xi, \lambda) \in M_k$.
\end{lemma}
The proof of this lemma has been established in Lemma 3 in \cite{chen2023disentangling}.
\begin{lemma}\label{Lemma_first}
If $(t, \xi, \lambda) \notin M_{k}$, then
$$
\left|T_{f_{k}}^{\left(t^{n} g\right)}(t, \xi, \lambda)\right| \leq \epsilon E_{k, n}(t),
$$
where $E_{k, n}(t):=\left\|\phi_{k}^{\prime}\right\|_{L^{\infty}} I_{n+1}+\left(D_{n}+\frac{\pi}{12}\left\|\phi_{k}^{\prime}\right\|_{L^{\infty}} I_{n+4}\right) A_{k}(t)$. 
\end{lemma}
\begin{proof}
First, we do the decomposition to $f_k(x)$ to evaluate the difference to it own at the time $t$
\begin{align*}
    f_{k}(x)=& A_{k}(x) e^{2 \pi i \phi_{k}(x)} \\
    =&\left(A_{k}(x)-A_{k}(t)\right) e^{2 \pi i \phi_{k}(x)}\\
    &+A_{k}(t) e^{2 \pi i\left[\phi_{k}(t)+\phi_{k}^{\prime}(t)(x-t)+\frac{1}{2} \phi_{k}^{\prime\prime}(t)(x-t)^{2}+\frac{1}{6}\phi_k^{\prime\prime \prime}(t)(x-t)^3\right]} \\
    &+A_{k}(t)\left(e^{2 \pi i \phi_{k}(x)}-e^{2 \pi i\left[\phi_{k}(t)+\phi_{k}^{\prime}(t)(x-t)+\frac{1}{2} \phi_{k}^{\prime\prime}(t)(x-t)^{2}+\frac{1}{6}\phi_k^{\prime\prime \prime}(t)(x-t)^3\right]}\right)\,.
\end{align*}
Denote these three terms in the sum by $f_{k,1}, f_{k,2}, f_{k,3}$ respectively and
$$ I_n = \int_\mathbb{R} |x-t|^n | g(x-t)|dx.$$
Then we have
\begin{align*}
    \left|T_{f_{k, 1}}^{\left(t^{n} g\right)}(t, \xi, \lambda)\right| 
    & \leq \int\left|A_{k}(x)-A_{k}(t)\right||x-t|^{n}|g(x-t)| \mathrm{d} x \\
    & \leq \left\|A_{k}^{\prime}\right\|_{L^{\infty}} \int|x-t|^{n+1}|g(x-t)| \mathrm{d} x \\
    & \leq \left\|A_{k}^{\prime}\right\|_{L^{\infty}} I_{n+1}\leq \epsilon\left\|\phi_k^\prime \right\|_{L^\infty}I_{n+1}\text{.}
\end{align*}
If $(t, \xi, \lambda) \notin M_{k}$, then
\begin{align*}
    &\left|T_{f_{k, 2}}^{\left(t^{n} g\right)}(t, \xi, \lambda)\right|\\
    &=\left|\int_{\mathbb{R}}A_{k}(t)e^{2\pi i \phi_k(t)}(x-t)^n g(x-t)\right.\\
    & \hspace{3cm}\times e^{-2\pi i (\xi-\phi_k'(t)) (x-t)-\pi i(\lambda-\phi_k''(t)) (x-t)^2+\frac{1}{3}\pi i\phi_k'''(t) (x-t)^3}dx\bigg| \\
    &=A_{k}(t)\left|\mathcal{C}\left(x^{{n}} g(x) e^{\frac{1}{3}\pi i \phi'''(t)x^3}\right)\left(\xi-\phi_{k}^{\prime}(t), \lambda-\phi_{k}^{\prime \prime}(t)\right)\right| \\
    & \leq A_{k}(t) D_{n} \epsilon\,\text{,}    
\end{align*}
providing $$\left|\mathcal{C}\left(x^{{n}} g(x) e^{\frac{1}{3}\pi i \phi_k'''(t)x^3}\right)(\xi, \eta)\right|\leq \frac{\sqrt{\Delta}D_n\epsilon}{\sqrt{|\xi|+|\eta|}}$$ where $\mathcal{C}(f(x))(\xi, \lambda):=\int_{\mathbb{R}}f(x)e^{-2\pi i \xi x-\pi i \lambda x^2}dx$. For the last term, we have
\begin{align*}
    &\quad\left|T_{f_{k, 3}}^{\left(t^{n} g\right)}(t, \xi, \lambda)\right|\\ 
    & \leq A_{k}(t) \int 2\pi \left|\phi_{k}(x)-\sum_{p=0}^3\frac{1}{p!}\phi_{k}^{(p)}(t)(x-t)^p\right||x-t|^{n}|g(x-t)| \mathrm{d} x \\
    & \leq A_{k}(t) \int \frac{2\pi }{4!}\left\|\phi_{k}^{(4)}\right\|_{L^{\infty}}|x-t|^{n+4}|g(x-t)| \mathrm{d} x \\
    & \leq \frac{\pi}{12} A_{k}(t)\left\|\phi_{k}^{(4)}\right\|_{L^{\infty}} I_{n+4}
    \leq\frac{\epsilon\pi}{12}A_k(t)\left\|\phi_k^\prime \right\|_{L^\infty}I_{n+4}\,\text{.}
\end{align*}
Consequently, 
$$\left|T_{f}^{\left(t^{n} g\right)}(t, \xi, \lambda)-T_{f_{k}}^{\left(t^{n} g\right)}(t, \xi, \lambda)\right| \leq \epsilon\sum_{l\neq k} E_{l, n}(t)\text{.}$$
\end{proof}
\begin{lemma}\label{Lemma_D1}
$\text { If }(t, \xi, \lambda) \in M_{k} \text { for some } k \text{ where } 1 \leq k \leq K, \text { then }$
$$\left|\partial_{t} T_{f}^{(g)}-2 \pi i\left(\phi_{k}^{\prime}(t) T_{f}^{(g)}+\phi_{k}^{\prime \prime}(t) T_{f}^{(t g)}+\frac{\phi'''_k(t)}{2}T^{(t^2g)}_f\right)\right| \leq  \epsilon B_{k, 1}(t)$$
where
\begin{align*}
B_{k,1}(t) &= \left(\sum_{k=1}^{N}\left\|\phi_{k}^{\prime}\right\|_{L^{\infty}} I_{0}+\frac{\pi}{3}\left\|\phi'_{k}\right\|_{L^{\infty}}\left\|A_{k}\right\|_{L^{\infty}} I_{3}\right.\\
&\quad+2 \pi \sum_{l \neq k}\left(\phi_{l}^{\prime}(t) E_{l, 0}(t)+\left|\phi_{l}^{\prime \prime}(t)\right| E_{l, 1}(t)+\frac{1}{2}|\phi'''_l(t)|E_{l,2}(t)\right) \\
&\quad\left.+2 \pi\left(\phi_{k}^{\prime}(t) \sum_{l \neq k} E_{l, 0}(t)+\left|\phi_{k}^{\prime \prime}(t)\right| \sum_{l \neq k} E_{l, 1}(t)+\frac{|\phi'''_k(t)|}{2}\sum_{l\neq k}E_{l,2}(t)\right)\right)\,.
\end{align*}
\end{lemma}
\begin{proof}
\begin{align*}
    \partial_{t} T_{f}^{(g)}=& \partial_{t}\left(\sum_{k=1}^{N} \int A_{k}(x) e^{2 \pi i \phi_{k}(x)} g(x-t) e^{-2 \pi i \xi(x-t)-\pi i \lambda(x-t)^{2}} \mathrm{~d} x\right) \\
    =&\sum_{k=1}^{N}  {\left[\int A_{k}^{\prime}(x) e^{2 \pi i \phi_{k}(x)} g(x-t) e^{-2 \pi i \xi(x-t)-\pi i \lambda(x-t)^{2}} \mathrm{~d} x\right.} \\
    &\hspace*{1cm}\left.+\int A_{k}(x) 2 \pi i \phi_{k}^{\prime}(x) e^{2 \pi i \phi_{k}(x)} g(x-t) e^{-2 \pi i \xi(x-t)-\pi i \lambda(x-t)^{2}} \mathrm{~d} x\right] \\
    =& \sum_{k=1}^{N}\left[\int A_{k}^{\prime}(x) e^{2 \pi i \phi_{k}(x)} g(x-t) e^{-2 \pi i \xi(x-t)-\pi i \lambda(x-t)^{2}} \mathrm{~d} x\right.\\
    &\hspace*{1cm}+\int A_{k}(x) 2 \pi i\left[\phi_{k}^{\prime}(t)+\phi_{k}^{\prime \prime}(t)(x-t)+\frac{1}{2}\phi'''_k(t)(x-t)^2\right.\\
    &\hspace*{1cm}\left.\left.+\frac{1}{6}\phi^{(4)}_k(\tau)(x-t)^3\right] 
     e^{2 \pi i \phi_{k}(x)} g(x-t) e^{-2 \pi i \xi(x-t)-\pi i \lambda(x-t)^{2}} \mathrm{~d} x\right]
\end{align*}
\begin{align*}
    &\quad\left|\partial_{t} T_{f}^{(g)}-2 \pi i \sum_{k=1}^{N}\left(\phi_{k}^{\prime}(t) T_{f_{k}}^{(g)}+\phi_{k}^{\prime \prime}(t) T_{f_{k}}^{(t g)}+\frac{\phi'''_k(t)}{2}T^{x^2g}_f\right)\right| \\
    \leq & \sum_{k=1}^{N}\left[\int\left|A_{k}^{\prime}(x) \| g(x-t)\right| \mathrm{d} x+2 \pi \int A_{k}(x)\left|\frac{1}{6}\phi^{(4)}_k(\tau)(x-t)^3\right||g(x-t)| \mathrm{d} x\right] \\
    \leq & \sum_{k=1}^{N}\left(\left\|A_{k}^{\prime}\right\|_{L^{\infty}} I_{0}+\frac{\pi}{3}\left\|\phi_{k}^{(4)}\right\|_{L^{\infty}}\left\|A_{k}\right\|_{L^{\infty}} I_{3}\right)  .
\end{align*}
By Lemma \ref{Lemma0} and \ref{Lemma_first} , when $(t, \xi, \lambda 
) \in M_{k}$, we have
$$
\begin{aligned}
&\quad\left|\partial_{t} T_{f}^{(g)}-2 \pi i\left(\phi_{k}^{\prime}(t) T_{f_{k}}^{(g)}+\phi_{k}^{\prime \prime}(t) T_{f_{k}}^{(t g)}+\frac{\phi'''_k(t)}{2}T^{(t^2g)}_{f_k}\right)\right| \\
& \leq \sum_{k=1}^{N}\left(\left\|A_{k}^{\prime}\right\|_{L^{\infty}} I_{0}+\frac{\pi}{3}\left\|\phi_{k}^{(4)}\right\|_{L^{\infty}}\left\|A_{k}\right\|_{L^{\infty}} I_{3}\right)\\
&\quad+2 \pi \epsilon\sum_{l \neq k}\left(\phi_{l}^{\prime}(t) E_{l, 0}(t)+\left|\phi_{l}^{\prime \prime}(t)\right| E_{l, 1}(t)+\frac{1}{2}|\phi'''_l(t)|E_{l,2}\right)\text{.}
\end{aligned}
$$
Therefore, we have
$$
\begin{aligned}
&\quad\left|\partial_{t} T_{f}^{(g)}-2 \pi i\left(\phi_{k}^{\prime}(t) T_{f}^{(g)}+\phi_{k}^{\prime \prime}(t) T_{f}^{(t g)}+\frac{\phi'''_k(t)}{2}T^{(t^2g)}_{f}\right)\right| \\
&\leq \left(\sum_{k=1}^{N}\left\|A_{k}^{\prime}\right\|_{L^{\infty}} I_{0}+\frac{\pi}{3}\left\|\phi_{k}^{(4)}\right\|_{L^{\infty}}\left\|A_{k}\right\|_{L^{\infty}} I_{3}\right.\\
&\quad+2 \pi \epsilon\sum_{l \neq k}\left(\phi_{l}^{\prime}(t) E_{l, 0}(t)+\left|\phi_{l}^{\prime \prime}(t)\right| E_{l, 1}(t)+\frac{1}{2}|\phi'''_l(t)|E_{l,2}\right) \\
&\left.\quad+2 \pi\epsilon\left(\phi_{k}^{\prime}(t) \sum_{l \neq k} E_{l, 0}(t)+\left|\phi_{k}^{\prime \prime}(t)\right| \sum_{l \neq k} E_{l, 1}(t)+\frac{1}{2}|\phi'''_k(t)|\sum_{l\neq k}E_{l,2}(t)\right)\right)\\
&\leq \epsilon\left(\sum_{k=1}^{N}\left\|\phi_{k}^{\prime}\right\|_{L^{\infty}} I_{0}+\frac{\pi}{3}\left\|\phi'_{k}\right\|_{L^{\infty}}\left\|A_{k}\right\|_{L^{\infty}} I_{3}\right.\\
&\quad+2 \pi \sum_{l \neq k}\left(\phi_{l}^{\prime}(t) E_{l, 0}(t)+\left|\phi_{l}^{\prime \prime}(t)\right| E_{l, 1}(t)+\frac{1}{2}|\phi'''_l(t)|E_{l,2}\right) \\
&\left.\quad+2 \pi\left(\phi_{k}^{\prime}(t) \sum_{l \neq k} E_{l, 0}(t)+\left|\phi_{k}^{\prime \prime}(t)\right| \sum_{l \neq k} E_{l, 1}(t)+\frac{1}{2}|\phi'''_k(t)|\sum_{l\neq k}E_{l,2}(t)\right)\right)\text{.}
\end{aligned}
$$
\end{proof}
\begin{lemma}\label{Lemma_D2}
$\text { If }(t, \xi, \lambda) \in M_{k} \text { for some } k\text{ where } 1 \leq k \leq K, \text { then }$
\begin{align*}
    \left|\partial_{t t}^{2} T_{f}^{(g)}-2 \pi i\bigg(\phi_{k}^{\prime \prime}(t) T_{f}^{(g)}+\phi'''_k(t)T^{(tg)}_f+\phi_{k}^{\prime}(t) \partial_{t} T_{f}^{(g)}+\phi_{k}^{\prime \prime}(t) \partial_{t} T_{f}^{(t g)}\right.\\
    +\left.\left.\frac{1}{2}\phi'''_k(t)\partial_t T^{(t^2g)}_f\right)\right| \leq \epsilon B_{k, 2}(t)
\end{align*}
where
\begin{align*}
    B_{k, 2}(t) &= \Omega+2 \pi \sum_{l \neq k}\bigg(\left|\phi_{l}^{\prime \prime}(t)\right| E_{l, 0}(t)+\left|\phi_{l}^{\prime \prime\prime}(t)\right| E_{l, 1}(t)+\phi_{l}^{\prime}(t) F_{l, 0}(t)\\
    &\quad\hspace{2.5cm}+\left.\left|\phi_{l}^{\prime \prime}(t)\right| F_{l, 1}(t)+\frac{1}{2}\left|\phi'''_l(t)\right|F_{l,2}(t)\right)\\
    &\quad+2 \pi\left(\left|\phi_{k}^{\prime \prime}(t)\right| \sum_{l \neq k} E_{l, 0}(t)+\left|\phi_{k}^{\prime \prime\prime}(t)\right| \sum_{l \neq k} E_{l, 1}(t)+\phi_{k}^{\prime}(t) \sum_{l \neq k} F_{l, 0}(t)\right.\\
    &\hspace{1.5cm}\quad+\left.\left|\phi_{k}^{\prime \prime}(t)\right| \sum_{l \neq k} F_{l, 1}(t)+\frac{1}{2}\left|\phi_{k}^{\prime \prime\prime}(t)\right| \sum_{l \neq k} F_{l, 2}(t)\right)
\end{align*}
with
\begin{align*}
    \Omega &= \sum_{k=1}^{N} \left(\left\|\phi'_{k}\right\|_{L^{\infty}} +2 \pi  \left\|\phi_{k}^{\prime}\right\|^2_{L^{\infty}}\right)I_{0}+\pi \left\|\phi'_{k}\right\|_{L^{\infty}}\left\|A_{k}\right\|_{L^{\infty}} I_{2}\\
    &\quad+\frac{\pi}{3} \left(\left\|A_{k}^{\prime}\right\|_{L^{\infty}}+2 \pi\left\|\phi_{k}^{\prime}\right\|_{L^{\infty}} \left\|A_{k}\right\|_{L^{\infty}}\right) \left\|\phi'_k\right\|_{L^{\infty}} I_{3}\,\text{.} 
\end{align*}
For $n = 1, 2, 3$
\begin{align*}
   F_{k,n} &= \left(\left\|\phi_{k}^{\prime}\right\|_{L^{\infty}} I_{n}+\frac{\pi}{3}\left\|\phi'_{k}\right\|_{L^{\infty}}\left\|A_{k}\right\|_{L^{\infty}} I_{n+3}\right)\\
   &\quad+2\pi \left(\phi'_k(t)E_{k,n}+|\phi''_k(t)|E_{k,n+1}+\frac{1}{2}|\phi'''_k(t)|E_{k,n+2}\right)\,.
\end{align*}
\end{lemma}
The proof is also similar to Lemma 6 in \cite{chen2023disentangling} hence we left it to the reader. The difference is that we now need to do further estimates to 
$$\left|\partial_{t} T_{f}^{(t^2 g)}-\partial_{t} T_{f_{k}}^{(t^2 g)}\right| \leq \epsilon\sum_{l \neq k} F_{l, 2}(t)\,.$$
\begin{lemma}\label{Lemma_D3}
$\text { If }(t, \xi, \lambda) \in M_{k} \text { for some } k, 1 \leq k \leq K, \text { then }$
\begin{align*}
    &\left|\partial^3_{ttt}T^{(g)}_f-2\pi i \left(\phi'''_k(t)T^{(g)}_f+2\phi''_k(t)\partial_t T^{(g)}_f+2\phi'''_k(t)\partial_t T^{(tg)}_f\right.\right.\\
    &\left.\left.+\phi'_k(t)\partial^2_{tt}T^{(g)}_f+\phi''_k(t)\partial^2_{tt}T^{(tg)}_f+\frac{1}{2}\phi'''_k(t)\partial^2_{tt}T^{(t^2g)}_f\right)\right|<\epsilon B_{k,3}
\end{align*}
where
\begin{align*}
    B_{k,3} &= \sum_{k=1}^N\bigg( \left(\left\|\phi'_k\right\|_{L^{\infty}}+2\pi\left\|\phi'_k\right\|^2_{L^{\infty}}\right)I_0+2\pi \left\|\phi'_k\right\|_{L^{\infty}}\left\|A_k\right\|_{L^{\infty}}I_1
    \\
    &\quad\quad+2\pi \left\|\phi'_k\right\|_{L^{\infty}}\left(\left\|A'_k\right\|_{L^{\infty}}+2\pi \left\|\phi'_k\right\|_{L^{\infty}}\left\|A_k\right\|_{L^{\infty}}\right)I_2+ \frac{\pi}{3}\left\|\phi'_k\right\|_{L^{\infty}}\Phi_kI_3\bigg)\\
    &\quad+2\pi \sum_{l\neq k}^N\bigg(\left|\phi'''_l(t)\right|E_{l,0}(t)+2\left|\phi''_l(t)\right|F_{l,0}(t)+2\left|\phi'''_l(t)\right|F_{l,1}(t)\\
    &\hspace{2cm}+\left|\phi'_l(t)\right|G_{l,0}(t)+\left|\phi''_l(t)\right|G_{l,1}(t)    +\frac{1}{2}\left|\phi'''_l(t)\right|G_{l,2}(t)\bigg)\\
    &\quad+2\pi \sum_{l\neq k}^N\bigg(\left|\phi'''_k(t)\right|E_{l,0}(t)+2\left|\phi''_k(t)\right|F_{l,0}(t)+2\left|\phi'''_k(t)\right|F_{l,1}(t)\\
    &\hspace{2cm}+\left|\phi'_k(t)\right|G_{l,0}(t)+\left|\phi''_k(t)\right|G_{l,1}(t)    +\frac{1}{2}\left|\phi'''_k(t)\right|G_{l,2}(t)\bigg)
\end{align*}
with
\begin{align*}
    \Phi_k = \left\|A''_k\right\|_{L^{\infty}}+2\pi&\bigg(\left\|A'_k\right\|_{L^{\infty}}\left\|\phi'_k\right\|_{L^{\infty}}+\left\|\phi''_k\right\|_{L^{\infty}}\left\|A_k\right\|_{L^{\infty}}+\left\|\phi'_k\right\|_{L^{\infty}}\left\|A'_k\right\|_{L^{\infty}}\\
    &\hspace{0.5cm}+2\pi \left\|\phi'_k\right\|_{L^{\infty}}\left\|A_k\right\|_{L^{\infty}}\bigg)
\end{align*}
and
\begin{align*}
    G_{k,n}(t) &= 2 \pi \bigg(\left|\phi_{k}^{\prime \prime}(t)\right| E_{k, n}(t)+\left|\phi'''_k(t)\right|E_{k, n+1}(t)+\left|\phi_{k}^{\prime}(t)\right| F_{k, n}(t)\\
    &\hspace{1cm}+\left|\phi_{k}^{\prime \prime}(t)\right|  F_{k, n+1}(t)+\frac{\left|\phi'''_k(t)\right|}{2} F_{k, n+2}(t)\bigg)\\
    &\quad+\bigg(\left\|\phi'_{k}\right\|_{L^{\infty}}+2 \pi  \left\|\phi_{k}^{\prime}\right\|^2_{L^{\infty}}\bigg)I_{n}+\pi \left\|\phi'_{k}\right\|_{L^{\infty}}\left\|A_{k}\right\|_{L^{\infty}} I_{n+2}\\
    &\quad+\frac{\pi }{3} \bigg(\left\|A_{k}^{\prime}\right\|_{L^{\infty}}+2 \pi\left\|\phi_{k}^{\prime}\right\|_{L^{\infty}} \left\|A_{k}\right\|_{L^{\infty}}\bigg) \left\|\phi'_k\right\|_{L^{\infty}} I_{n+3}\,.
\end{align*}
\end{lemma}
\begin{proof}
If $f_k(x)= A_k(x)e^{2\pi i \phi_k (x)}$, then one has
\begin{align*}
    f'''_k(x)=& A'''_k(x)e^{2\pi i \phi_k (x)}+A''_k(x)2\pi i\phi'_k(x)e^{2\pi i \phi_k (x)}
    +2\pi i \phi'''_k(x)f_k(x)\\
    &+4\pi i\phi''_k(x)f'_k(x)+2\pi i\phi_k(x)f''_k(x).
\end{align*}
Denote the five terms in the right-handed side by $f_{k,1}, f_{k,2}, \cdots ,f_{k,5}$, respectively, and hence
$$f(x)=\sum_{k=1}^N f_k(x),\ f'''(x)=\sum_{k=1}^N f_{k,1}(x)+f_{k,2}(x)+\cdots+f_{k,5}(x).$$
Denote the kernel by
$$\mathcal{K}(x-t):=e^{-2 \pi i \xi(x-t)-\pi i \lambda(x-t)^{2}},$$
The first and second terms give
\begin{align*}
    \left|T^{(g)}_{f_{k,1}}\right|=\left|\int A'''_k(x)e^{2\pi i \phi_k (x)}g(x-t)\mathcal{K}(x-t) dx\right|\leq \left\|A'''_k\right\|_{L^{\infty}}I_0\leq\epsilon\left\|\phi'_k\right\|_{L^{\infty}}I_0
\end{align*}
\begin{align*}
    \left|T^{(g)}_{f_{k,2}}\right|&=\left|\int A''_k(x)2\pi i \phi'_k(x)e^{2\pi i \phi_k(x)}g(x-t)\mathcal{K}(x-t)dx\right|\\
    &\leq 2\pi\left\|\phi'_k\right\|_{L^{\infty}}\left\|A''_k\right\|_{L^{\infty}} I_0\\
    &\leq 2\pi\epsilon\left\|\phi'_k\right\|^2_{L^{\infty}} I_0,
\end{align*}
the third term yields
\begin{align*}
    \left|T^{(g)}_{f_{k,3}}-2\pi i\phi'''_k(t)T^{(g)}_{f_k}\right|&=2\pi\left|\int \left(\phi'''_k(x)-\phi'''_k(t)\right)f_k(x)g(x-t)\mathcal{K}(x-t)dx\right|\\
    &\leq 2\pi \left\|\phi^{(4)}_k\right\|_{L^{\infty}}\left\|A_k\right\|_{L^{\infty}}I_1\leq 2\pi\epsilon\left\|\phi'_k\right\|_{L^{\infty}}\left\|A_k\right\|_{L^{\infty}}I_1,
\end{align*}
the fourth term yields
\begin{align*}
    &\left|T^{(g)}_{f_{k,4}}-4\pi i\left(\phi''_k(t) \partial_t T^{(g)}_{f_k}+\phi'''_k(t)\partial_t T^{(tg)}_{f_k}\right)\right|\\
    &=4\pi\left|\int \left(\phi''_k(x)-\phi''_k(t)-\phi'''(t)(x-t)\right)f'_k(x)g(x-t)\mathcal{K}(x-t)dx\right|\\
    &\leq \left|\int \left(\frac{1}{2}\phi^{(4)}_k(\tau)(x-t)^2\right)f'_k(x)g(x-t)\mathcal{K}(x-t)dx\right|\\
    &\leq 2\pi \left\|\phi^{(4)}_k\right\|_{L^{\infty}}\left(\left\|A'_k\right\|_{L^{\infty}}+2\pi \left\|\phi'_k\right\|_{L^{\infty}}\left\|A_k\right\|_{L^{\infty}}\right)I_2\\
    &\leq 2\pi\epsilon \left\|\phi'_k\right\|_{L^{\infty}}\left(\left\|A'_k\right\|_{L^{\infty}}+2\pi \left\|\phi'_k\right\|_{L^{\infty}}\left\|A_k\right\|_{L^{\infty}}\right)I_2,
\end{align*}
and the last term leads to
\begin{align*}
    &\left|T^{(g)}_{f_{k,5}}-2\pi i \left(\phi'_k(t)\partial^2_{tt}T^{(g)}_{f_k}+\phi''_k(t)\partial^2_{tt}T^{(tg)}_{f_k}+\frac{1}{2}\phi'''_k(t)\partial^2_{tt}T^{(t^2g)}_{f_k}\right)\right|\\
    &=2\pi \left|\int \left( \frac{1}{3!}\phi^{(4)}_k(\tau)(x-t)^3\right)f''_k(x)g(x-t)\mathcal{K}(x-t)dx\right|\\
    &\leq \frac{\pi}{3}\left\|\phi^{(4)}_k\right\|_{L^{\infty}}\left\|f''_k\right\|_{L^{\infty}}I_3\leq \frac{\pi\epsilon}{3}\left\|\phi'_k\right\|_{L^{\infty}}\left\|f''_k\right\|_{L^{\infty}}I_3
\end{align*}
where
\begin{align*}
\left\|f''_{k}\right\|_{L^{\infty}}\leq \left\|A''_k\right\|_{L^{\infty}}+2\pi&\bigg(\left\|A'_k\right\|_{L^{\infty}}\left\|\phi'_k\right\|_{L^{\infty}}+\left\|\phi''_k\right\|_{L^{\infty}}\left\|A_k\right\|_{L^{\infty}}\\
&\quad+\left\|\phi'_k\right\|_{L^{\infty}}\left\|A'_k\right\|_{L^{\infty}}+2\pi \left\|\phi'_k\right\|_{L^{\infty}}\left\|A_k\right\|_{L^{\infty}}\bigg)=\Phi_k
\end{align*}
and 
\begin{align*}
    &\left|\partial^3_{ttt}T^{(g)}_f-2\pi i \sum_{k=1}^N\bigg(\phi'''_k(t)T^{(g)}_{f_k}+2\phi''_k(t)\partial_t T^{(g)}_{f_k}+2\phi'''_k(t)\partial_t T^{(tg)}_{f_k}\right.\\
    &\hspace{3.5cm}\left.+\phi'_k(t)\partial^2_{tt}T^{(g)}_{f_k}+\phi''_k(t)\partial^2_{tt}T^{(tg)}_{f_k}+\frac{1}{2}\phi'''_k(t)\partial^2_{tt}T^{(t^2g)}_{f_k}\bigg)\right|\\
    &\leq \epsilon\sum_{k=1}^N \bigg((\left\|\phi'_k\right\|_{L^{\infty}}+2\pi\left\|\phi'_k\right\|^2_{L^{\infty}})I_0+2\pi \left\|\phi'_k\right\|_{L^{\infty}}\left\|A_k\right\|_{L^{\infty}}I_1\\
    &\quad\quad\quad+2\pi \left\|\phi'_k\right\|_{L^{\infty}}\left(\left\|A'_k\right\|_{L^{\infty}}+2\pi \left\|\phi'_k\right\|_{L^{\infty}}\left\|A_k\right\|_{L^{\infty}}\right)I_2+ \frac{\pi}{3}\left\|\phi'_k\right\|_{L^{\infty}}\Phi_kI_3\bigg)\,.
\end{align*}
Applying the proof in Lemma \ref{Lemma_D2},
\begin{align*}
    &\bigg|\partial_{t t}^{2} T_{f_k}^{(t^ng)}-2 \pi i\bigg(\phi_{k}^{\prime \prime}(t) T_{f_k}^{(t^ng)}+\phi'''_k(t)T^{(t^{n+1}g)}_{f_k}+\phi_{k}^{\prime}(t) \partial_{t} T_{f_k}^{(t^ng)}\\
    &\hspace{3.5cm}+\phi_{k}^{\prime \prime}(t) \partial_{t} T_{f_k}^{(t^{n+1} g)}+\frac{1}{2}\phi'''_k(t)\partial_t T^{(t^{n+2}g)}_{f_k}\bigg)  \bigg|\\
     &\leq \epsilon\left(\left\|\phi'_{k}\right\|_{L^{\infty}} +2 \pi  \left\|\phi_{k}^{\prime}\right\|^2_{L^{\infty}}\right)I_{n}+\pi \epsilon\left\|\phi'_{k}\right\|_{L^{\infty}}\left\|A_{k}\right\|_{L^{\infty}} I_{n+2}\\
     &\quad+\frac{\pi\epsilon }{3} \bigg(\left\|A_{k}^{\prime}\right\|_{L^{\infty}}+2 \pi\left\|\phi_{k}^{\prime}\right\|_{L^{\infty}} \left\|A_{k}\right\|_{L^{\infty}}\bigg) \left\|\phi'_k\right\|_{L^{\infty}} I_{n+3}\, \text{.}
\end{align*}
This implies that, combining with Lemma \ref{Lemma_first}, for $n=0,1,2,3...$
\begin{align*}
    &\quad\left|\partial_{t t}^{2} T_{f_k}^{(t^ng)} \right|\\
    &\leq 2 \pi  \epsilon\bigg(\left|\phi_{k}^{\prime \prime}(t)\right| E_{k, n}(t)+\left|\phi'''_k(t)\right|E_{k, n+1}(t)+\left|\phi_{k}^{\prime}(t)\right|  F_{k, n}(t)\\
    &\hspace{1.3cm}+\left|\phi_{k}^{\prime \prime}(t)\right|  F_{k, n+1}(t)+\frac{\left|\phi'''_k(t)\right|}{2} F_{k, n+2}(t)\bigg)\\
    &\quad+\epsilon\bigg(\left\|\phi'_{k}\right\|_{L^{\infty}}+2 \pi  \left\|\phi_{k}^{\prime}\right\|^2_{L^{\infty}}\bigg)I_{n}+\pi \epsilon\left\|\phi'_{k}\right\|_{L^{\infty}}\left\|A_{k}\right\|_{L^{\infty}} I_{n+2}\\
    &\quad+\frac{\pi\epsilon }{3} \left(\left\|A_{k}^{\prime}\right\|_{L^{\infty}}+2 \pi\left\|\phi_{k}^{\prime}\right\|_{L^{\infty}} \left\|A_{k}\right\|_{L^{\infty}}\right) \left\|\phi'_k\right\|_{L^{\infty}} I_{n+3}\\
    &= \epsilon G_{k,n}(t)\,.
\end{align*}
Therefore, for $n=0,1,2,3...$
\begin{align*}
    \left|\partial_{tt}^2 T_{f}^{(t^ng)}-\partial_{tt}^2 T_{f_{k}}^{(t^ng)}\right| \leq \epsilon\sum_{l \neq k} G_{l, n}(t).
\end{align*}
Consequently, as $(t, \xi, \lambda ) \in M_k$,
\begin{align*}
    &\quad\bigg|\partial^3_{ttt}T^{(g)}_f-2\pi i \bigg(\phi'''_k(t)T^{(g)}_f+2\phi''_k(t)\partial_t T^{(g)}_f+2\phi'''_k(t)\partial_t T^{(tg)}_f+\phi'_k(t)\partial^2_{tt}T^{(g)}_f\\
    &\hspace{3cm}+\phi''_k(t)\partial^2_{tt}T^{(tg)}_f+\frac{1}{2}\phi'''_k(t)\partial^2_{tt}T^{(t^2g)}_f\bigg)\bigg|\\
    &\leq \left|\partial^3_{ttt}T^{(g)}_f-2\pi i \sum_{k=1}^N\bigg(\phi'''_k(t)T^{(g)}_{f_k}+2\phi''_k(t)\partial_t T^{(g)}_{f_k}+2\phi'''_k(t)\partial_t T^{(tg)}_{f_k}\right.\\
    &\hspace{4cm}\left.\left.+\phi'_k(t)\partial^2_{tt}T^{(g)}_{f_k}
    +\phi''_k(t)\partial^2_{tt}T^{(tg)}_{f_k}+\frac{1}{2}\phi'''_k(t)\partial^2_{tt}T^{(t^2g)}_{f_k}\right)\right|\\
    &\quad+\left|2\pi i\sum_{l\neq k}^N\bigg(\phi'''_l(t)T^{(g)}_{f_l}+2\phi''_l(t)\partial_t T^{(g)}_{f_l}+2\phi'''_l(t)\partial_t T^{(tg)}_{f_l}+\phi'_l(t)\partial^2_{tt}T^{(g)}_{f_l}\right.\\
    &\hspace{2.5cm}+\phi''_l(t)\partial^2_{tt}T^{(tg)}_{f_l}
    \left.\left.+\frac{1}{2}\phi'''_l(t)\partial^2_{tt}T^{(t^2g)}_{f_l}\right)\right|\\
    &\quad+\bigg|2\pi i \left(\phi'''_k(t)\left(T^{(g)}_f-T^{(g)}_{f_k}\right)+2\phi''_k(t)\left(\partial_t T^{(g)}_f-\partial_t T^{(g)}_{f_k}\right)\right.\\
    &\hspace{2cm}\left.+2\phi'''_k(t)\left(\partial_t T^{(tg)}_f-\partial_t T^{(tg)}_{f_k}\right)+\phi'_k(t)\left(\partial^2_{tt}T^{(g)}_f-\partial^2_{tt}T^{(g)}_{f_k}\right)\right.\\
    &\hspace{2cm}\left.+\phi''_k(t)\left(\partial^2_{tt}T^{(tg)}_f-\partial^2_{tt}T^{(tg)}_{f_k}\right)+\frac{1}{2}\phi'''_k(t)\left(\partial^2_{tt}T^{(t^2g)}_f-\partial^2_{tt}T^{(t^2g)}_{f_k}\right)\bigg)\right| \\
    & := J_1+J_2+J_3
    \end{align*}
    where
    \begin{align*}
    J_1&= \left|\partial^3_{ttt}T^{(g)}_f-2\pi i \sum_{k=1}^N\bigg(\phi'''_k(t)T^{(g)}_{f_k}+2\phi''_k(t)\partial_t T^{(g)}_{f_k}+2\phi'''_k(t)\partial_t T^{(tg)}_{f_k}\right.\\
    &\hspace{4cm}\left.\left.+\phi'_k(t)\partial^2_{tt}T^{(g)}_{f_k}
    +\phi''_k(t)\partial^2_{tt}T^{(tg)}_{f_k}+\frac{1}{2}\phi'''_k(t)\partial^2_{tt}T^{(t^2g)}_{f_k}\right)\right|\\
    &\leq \epsilon\sum_{k=1}^N\bigg( \left(\left\|\phi'_k\right\|_{L^{\infty}}+2\pi\left\|\phi'_k\right\|^2_{L^{\infty}}\right)I_0+2\pi \left\|\phi'_k\right\|_{L^{\infty}}\left\|A_k\right\|_{L^{\infty}}I_1\\
    &\quad\quad+2\pi \left\|\phi'_k\right\|_{L^{\infty}}\left(\left\|A'_k\right\|_{L^{\infty}}+2\pi \left\|\phi'_k\right\|_{L^{\infty}}\left\|A_k\right\|_{L^{\infty}}\right)I_2+\frac{\pi}{3}\left\|\phi'_k\right\|_{L^{\infty}}\Phi_kI_3\bigg)\,,\\
    J_2& = \left|2\pi i\sum_{l\neq k}^N\bigg(\phi'''_l(t)T^{(g)}_{f_l}+2\phi''_l(t)\partial_t T^{(g)}_{f_l}+2\phi'''_l(t)\partial_t T^{(tg)}_{f_l}+\phi'_l(t)\partial^2_{tt}T^{(g)}_{f_l}\right.\\
    &\hspace{2.5cm}+\phi''_l(t)\partial^2_{tt}T^{(tg)}_{f_l}
    \left.\left.+\frac{1}{2}\phi'''_l(t)\partial^2_{tt}T^{(t^2g)}_{f_l}\right)\right|\\
    &\leq 2\pi \epsilon\sum_{l\neq k}^N\bigg(\left|\phi'''_l(t)\right|E_{l,0}(t)+2\left|\phi''_l(t)\right|F_{l,0}(t)+2\left|\phi'''_l(t)\right|F_{l,1}(t)\\
    &\hspace{2cm}+\left|\phi'_l(t)\right|G_{l,0}(t)+\left|\phi''_l(t)\right|G_{l,1}(t)    +\frac{1}{2}\left|\phi'''_l(t)\right|G_{l,2}(t)\bigg)\,,
    \end{align*}
    and 
    \begin{align*}
    J_3& = \bigg|2\pi i \left(\phi'''_k(t)\left(T^{(g)}_f-T^{(g)}_{f_k}\right)+2\phi''_k(t)\left(\partial_t T^{(g)}_f-\partial_t T^{(g)}_{f_k}\right)\right.\\
    &\hspace{2cm}\left.+2\phi'''_k(t)\left(\partial_t T^{(tg)}_f-\partial_t T^{(tg)}_{f_k}\right)+\phi'_k(t)\left(\partial^2_{tt}T^{(g)}_f-\partial^2_{tt}T^{(g)}_{f_k}\right)\right.\\
    &\hspace{2cm}\left.+\phi''_k(t)\left(\partial^2_{tt}T^{(tg)}_f-\partial^2_{tt}T^{(tg)}_{f_k}\right)+\frac{1}{2}\phi'''_k(t)\left(\partial^2_{tt}T^{(t^2g)}_f-\partial^2_{tt}T^{(t^2g)}_{f_k}\right)\bigg)\right| \\
    &\leq 2\pi \epsilon\sum_{l\neq k}^N\bigg(\left|\phi'''_k(t)\right|E_{l,0}(t)+2\left|\phi''_k(t)\right|F_{l,0}(t)+2\left|\phi'''_k(t)\right|F_{l,1}(t)\\
    &\hspace{2cm}+\left|\phi'_k(t)\right|G_{l,0}(t)+\left|\phi''_k(t)\right|G_{l,1}(t)    +\frac{1}{2}\left|\phi'''_k(t)\right|G_{l,2}(t)\bigg)\,.
\end{align*}
Combine the above three estimates to conclude the proof of this lemma.
\end{proof}
\begin{lemma}\label{Key_lemma}
For any $(t, \xi, \lambda) \in M_k$ such that $\left|T_f^{(g)}(t,\xi,\lambda)\right| > \Tilde{\epsilon}$, $ \left|q^0_f(t,\xi,\lambda)\right|>\Tilde{\epsilon}$ and $\pi\left|2+\partial_t\left(\dfrac{q^1_f}{q^0_f}\right)\right|>\Tilde{\epsilon}$ , we have
$$|\theta_f^{(g)}(t,\xi,\lambda) - \phi'''_k(t)| \leq \Tilde{\epsilon} $$
provided $\epsilon$ is sufficiently small.
\end{lemma}
\begin{proof}
Let $Q = 2+\partial_t\left(\frac{2\frac{T^{(tg)}_f}{T^{(g)}_f}+\partial_t\left(\frac{T^{(t^2g)}_f}{T^{(g)}_f}\right)}{1+\partial_t\left(\frac{T^{(tg)}_f}{T^{(g)}_f}\right)}\right)=2+\partial_t\left(\dfrac{q^1_f}{q^0_f}\right)$ and $N = \partial_t\left(   \frac{\partial_t \left(\frac{\partial_t T^{(g)}_f }{T^{(g)}_f }\right) }{1+\partial_t\left(\frac{T^{(tg)}_f}{T^{(g)}_f}\right)}  \right)$
\begin{align*} 
&|\theta_f^{(g)}(t,\xi,\lambda) - \phi_k^{\prime\prime\prime}(t)|\\
&=\left|\phi_k^{\prime\prime\prime}(t) - \frac{1}{\pi i}\partial_t\left(   \frac{\partial_t \left(\frac{\partial_t T^{(g)}_f }{T^{(g)}_f }\right) }{1+\partial_t\left(\frac{T^{(tg)}_f}{T^{(g)}_f}\right)}  \right)\left[2+\partial_t\left(\frac{2\frac{T^{(tg)}_f}{T^{(g)}_f}+\partial_t\left(\frac{T^{(t^2g)}_f}{T^{(g)}_f}\right)}{1+\partial_t\left(\frac{T^{(tg)}_f}{T^{(g)}_f}\right)}\right)\right]^{-1} \right|\\
&=\dfrac{1}{\pi |Q|}\left|N - \pi i Q \phi'''_k(t)\right|.
\end{align*}
Recall that $q^0_f = 1+\partial_t\left(\frac{T^{tg}_f}{T^g_f}\right)$, so we naturally have $$(T^g_f)^2q^0_f = \left((T^{(g)}_f)^2+\partial_t T^{(tg)}_f T^{(g)}_f-T^{(tg)}_f\partial_t T^{(g)}_f\right)$$. After simplifying and rewriting the equation in an appropriate order, at $(t, \xi, \lambda) \in M_k$, it will become 
\begin{align*}
    &\quad\pi |Q|\bigg|\theta_f^{(g)}(t,\xi,\lambda) - \phi_k^{\prime\prime\prime}(t)\bigg|\\
    =& \bigg| \frac{1}{T^g_f q^0_f} \bigg[\partial^3_{ttt}T^{(g)}_f-2\pi i \bigg(\phi'''_k(t)T^{(g)}_f+2\phi''_k(t)\partial_t T^{(g)}_f+2\phi'''_k(t)\partial_t T^{(tg)}_f\\
    &\hspace{3.8cm}+\phi'_k(t)\partial^2_{tt}T^{(g)}_f+\phi''_k(t)\partial^2_{tt}T^{(tg)}_f+\frac{1}{2}\phi'''_k(t)\partial^2_{tt}T^{(t^2g)}_f\bigg)\bigg]\\
    &\quad+\frac{1}{(T^g_f)^3(q^0_f)^2}\bigg[\partial_{t t}^{2} T_{f}^{(g)}-2 \pi i\bigg(\phi_{k}^{\prime \prime}(t) T_{f}^{(g)}+\phi'''_k(t)T^{(tg)}_f+\phi_{k}^{\prime}(t) \partial_{t} T_{f}^{(g)}\\
    &\hspace{5.5cm}+\phi_{k}^{\prime \prime}(t) \partial_{t} T_{f}^{(t g)}+\frac{1}{2}\phi'''_k(t)\partial_t T^{(t^2g)}_f\bigg)\bigg]\\
    &\quad\quad\times\bigg(T^{(tg)}_f\partial^2_{tt} T^{(g)}_f-T^{(g)}_f\bigg(2\partial_t T^{(g)}_f-\partial^2_{tt}T^{(tg)}_f\bigg)\\
    &\quad\quad-\bigg[\partial_{t} T_{f}^{(g)}-2 \pi i\left(\phi_{k}^{\prime}(t) T_{f}^{(g)}+\phi_{k}^{\prime \prime}(t) T_{f}^{(t g)}+\frac{\phi'''_k(t)}{2}T^{(t^2g)}_f\right)\bigg]\partial_t T^{(tg)}_f\bigg)\\
    &\quad-\frac{1}{(T^g_f)^3(q^0_f)^2}\bigg[\partial_{t} T_{f}^{(g)}-2 \pi i\left(\phi_{k}^{\prime}(t) T_{f}^{(g)}+\phi_{k}^{\prime \prime}(t) T_{f}^{(t g)}+\frac{\phi'''_k(t)}{2}T^{(t^2g)}_f\right)\bigg]\\
    &\quad\quad\times\bigg(\partial^2_{tt}T^{(g)}_f\left(T^{(g)}_f+\partial_t T^{(tg)}_f\right)    -\partial_t T^{(g)}_f\partial^2_{tt}T^{(tg)}_f-2(\partial_t T^{(g)}_f)^2\\
    &\quad\quad-\bigg[\partial_{t t}^{2} T_{f}^{(g)}-2 \pi i\bigg(\phi_{k}^{\prime \prime}(t) T_{f}^{(g)}+\phi'''_k(t)T^{(tg)}_f+\phi_{k}^{\prime}(t) \partial_{t} T_{f}^{(g)}+\phi_{k}^{\prime \prime}(t) \partial_{t} T_{f}^{(t g)}\\
    &\hspace{4cm}+\frac{1}{2}\phi'''_k(t)\partial_t T^{(t^2g)}_f\bigg)\bigg]\left(T^{(g)}_f+\partial_t T^{(tg)}_f\right)\bigg)\bigg|\,.
    \end{align*}
    Applying the previous lemmas to dominate these three terms, we have
    \begin{align*}
    &\quad\pi |Q|\bigg|\theta_f^{(g)}(t,\xi,\lambda) - \phi_k^{\prime\prime\prime}(t)\bigg|\\
    &\leq \frac{\epsilon B_{k,3}}{|T^g_f q^0_f|}\\
    &\quad+\frac{\epsilon B_{k,2}\left(\left|T^{(tg)}_f\partial^2_{tt} T^{(g)}_f\right|+\epsilon B_{k,1}\left|\partial_t T^{(tg)}_f\right|+\left|T^{(g)}_f\left(2\partial_t T^{(g)}_f-\partial^2_{tt}T^{(tg)}_f\right)\right|\right)}{|T^g_f|^3(q^0_f)^2}\\
    &\quad+\frac{\epsilon B_{k,1}\bigg(\left|\partial^2_{tt}T^{(g)}_f\left(T^{(g)}_f+\partial_tT^{(tg)}_f\right)\right|+\left|\partial_tT^{(g)}_f\partial^2_{tt}T^{(tg)}_f\right|+2(\partial_tT^{(g)}_f)^2}{|T^g_f|^3(q^0_f)^2}\\
    &\hspace{2cm}\frac{+\epsilon B_{k,2}\left|\left(T^{(g)}_f+\partial_tT^{(tg)}_f\right)\right|\bigg)}{|T^g_f|^3(q^0_f)^2}\,.
\end{align*}
With the uniform boundedness of $B_{k, n}(t),\left|\partial^2_{tt} T_f^{(t^ng)}\right|, \left|\partial_t T_f^{(t^ng)}\right|$ and $\left|T_f^{(t^ng)}\right|$, we also get the boundedness of $\left|q^0_f\right|$. If we impose an extra restriction on $\epsilon$, namely that, for all $k \in\{1, \ldots, K\}$ and all $(t, \xi, \lambda) \in M_k$,
\begin{equation*}
    \left[\begin{aligned}
    &\quad(T^{(g)})^2|q^0_f|B_{k,3}\\
    &+B_{k,2}\left(\left|T^{(tg)}_f\partial^2_{tt} T^{(g)}_f\right|+\epsilon B_{k,1}\left|\partial_t T^{(tg)}_f\right|+\left|T^{(g)}_f\left(2\partial_t T^{(g)}_f-\partial^2_{tt}T^{(tg)}_f\right)\right|\right)\\
    &+B_{k,1}\bigg(\left|\partial^2_{tt}T^{(g)}_f\left(T^{(g)}_f+\partial_tT^{(tg)}_f\right)\right|+\left|\partial_tT^{(g)}_f\partial^2_{tt}T^{(tg)}_f\right|+2(\partial_tT^{(g)}_f)^2\\
    &\hspace{1cm}+\epsilon B_{k,2}\left|\left(T^{(g)}_f+\partial_tT^{(tg)}_f\right)\right|\bigg)
    \end{aligned}\right]^{-1}
    \geq \Tilde{\epsilon}^{3},
\end{equation*}
then $|\theta_f^{(g)}(t,\xi,\lambda) - \phi'''_k(t)| \leq \Tilde{\epsilon}$.
\end{proof}
\begin{lemma}
For any $(t, \xi, \lambda) \in M_k$ such that $\left|T_f^{(g)}(t,\xi,\lambda)\right| > \Tilde{\epsilon}$, $ \left|q^0_f(t,\xi,\lambda)\right|>\Tilde{\epsilon}$ and $\pi\left|2+\partial_t\left(\dfrac{q^1_f}{q^0_f}\right)\right|>\Tilde{\epsilon}$ , we have
$$|\mu_f^{(g)}(t,\xi,\lambda) - \phi_k^{\prime\prime}(t)| \leq \Tilde{\epsilon} $$
provided $\epsilon$ is sufficiently small.
\end{lemma}

\begin{proof}
By the previous lemmas, we have at $(t, \xi, \lambda) \in M_k$
\begin{align*}
    &\quad\left|\mu^{(g)}_f(t, \xi, \lambda)-\phi''_k(t)\right|\\
    &=\left|\frac{\partial_t \left(\frac{\partial_t T^{(g)}_f }{T^{(g)}_f }\right)-\pi i \theta^{(g)}_f(t, \xi, \lambda)\left[2\frac{T^{(tg)}_f}{T^{(g)}_f}+\partial_t\left(\frac{T^{(t^2g)}_f}{T^{(g)}_f}\right)\right]}{2\pi i\left(1+\partial_t\left(\frac{T^{(tg)}_f}{T^{(g)}_f}\right)\right)}-\phi''_k(t)\right| \\
    &=\left|\frac{T^{(g)}_f\bigg[\partial^2_{tt}T^{(g)}_f-2\pi i\bigg(\phi_{k}^{\prime \prime}(t) T_{f}^{(g)}+\phi'''_k(t)T^{(tg)}_f+\phi_{k}^{\prime}(t) \partial_{t} T_{f}^{(g)}}{2\pi i\left((T^{(g)}_f)^2+\partial_t T^{(tg)}_f T^{(g)}_f-T^{(tg)}_f\partial_t T^{(g)}_f\right)}\right.\\
    &\hspace{4.5cm}+\frac{\phi_{k}^{\prime \prime}(t) \partial_{t} T_{f}^{(t g)}+\frac{1}{2}\phi'''_k(t)\partial_t T^{(t^2g)}_f\bigg)\bigg]}{2\pi i\left((T^{(g)}_f)^2+\partial_t T^{(tg)}_f T^{(g)}_f-T^{(tg)}_f\partial_t T^{(g)}_f\right)}\\
    &\quad\left.-\frac{\partial_t T^{(g)}_f\left[\partial_{t} T_{f}^{(g)}-2 \pi i\left(\phi_{k}^{\prime}(t) T_{f}^{(g)}+\phi_{k}^{\prime \prime}(t) T_{f}^{(t g)}+\frac{\phi'''_k(t)}{2}T^{(t^2g)}_f\right)\right]}{2\pi i\left((T^{(g)}_f)^2+\partial_t T^{(tg)}_f T^{(g)}_f-T^{(tg)}_f\partial_t T^{(g)}_f\right)}\right.\\
    &\quad+\left.\frac{\pi i \left(\phi'''_k(t)-\theta^{(g)}_f\right)\left(T^{(g)}_f\right)^2\left[2\frac{T^{(tg)}_f}{T^{(g)}_f}+\partial_t\left(\frac{T^{(t^2g)}_f}{T^{(g)}_f}\right)\right]}{2\pi i\left((T^{(g)}_f)^2+\partial_t T^{(tg)}_f T^{(g)}_f-T^{(tg)}_f\partial_t T^{(g)}_f\right)}\right|\\
    &\leq\frac{\epsilon}{2\pi |T^g_fq^0_f|}B_{k,2}+\frac{\epsilon}{2\pi (T^g_f)^2|q^0_f|}B_{k,1}\left|\partial_t T^{(g)}_f\right| +\frac{\epsilon}{2 |q^0_f|}\left|q^1_f\right|,
\end{align*}
Also with the boundedness of $\left|q^1_f\right|$ and the similar argument in the proof of Lemma \ref{Key_lemma}, $|\mu_f^{(g)}(t,\xi,\lambda) - \phi_k^{\prime\prime}(t)| \leq \Tilde{\epsilon}$ if $\epsilon$ is sufficient small.
\end{proof}
\begin{lemma}
For any $(t, \xi, \lambda) \in M_k$ such that $\left|T_f^{(g)}(t,\xi,\lambda)\right| > \Tilde{\epsilon}$, $ \left|q^0_f(t,\xi,\lambda)\right|>\Tilde{\epsilon}$ and $\pi\left|2+\partial_t\left(\dfrac{q^1_f}{q^0_f}\right)\right|>\Tilde{\epsilon}$ , we have
$$|\omega_f^{(g)}(t,\xi,\lambda) - \phi_k^{\prime}(t)| \leq \Tilde{\epsilon} $$
provided $\epsilon$ is sufficiently small.
\end{lemma}

\begin{proof}
By the previous lemmas, we have at $(t, \xi, \lambda) \in M_k$
\begin{align*}
&\quad|\omega^{(g)}_f(t, \xi, \lambda)-\phi^\prime_k(t)|\\
&=\left|\frac{1}{2\pi i}\left(\frac{\partial_t T^{(g)}_f }{T^{(g)}_f }\right)-\mu^{(g)}_f(t, \xi, \lambda)\left(\frac{T^{(tg)}_f}{T^{(g)}_f}\right)-\frac{1}{2}\theta^{(g)}_f(t, \xi, \lambda)\left(\frac{T^{(t^2g)}_f}{T^{(g)}_f}\right) - \phi^\prime_k(t)\right|\\
&=\left|\frac{\partial_t T_f^{(g)} - 2\pi i\mu^{(g)}_f(t, \xi, \lambda)T_f^{(tg)} - \pi i\theta^{(g)}_f(t, \xi, \lambda)T_f^{(t^2g)} - 2\pi i\phi^\prime_k(t)T_f^{(g)}}{2\pi i T_f^{(g)}}\right|\\
&=\left|\frac{\bigg[\partial_{t} T_{f}^{(g)}-2 \pi i\left(\phi_{k}^{\prime}(t) T_{f}^{(g)}+\phi_{k}^{\prime \prime}(t) T_{f}^{(t g)}+\frac{\phi'''_k(t)}{2}T^{(t^2g)}_f\right)\bigg]}{2\pi i T_f^{(g)}}\right.\\
&\quad\left.+\frac{(\phi^{\prime\prime}_k(t)- \mu_f^{(g)}(t,\xi,\lambda))T_f^{(tg)}}{T_f^{(g)}}
+\frac{(\phi^{\prime\prime\prime}_k(t)- \theta_f^{(g)}(t,\xi,\lambda))T_f^{(t^2g)}}{2T_f^{(g)}}\right|\\
&\leq \frac{\epsilon}{2\pi|T_f^{(g)}|}B_{k,1}+\frac{\epsilon}{|T_f^{(g)}|}\left|T_{f}^{(t g)}\right|+\frac{\epsilon}{2|T_f^{(g)}|}\left|T_f^{(t^2g)}\right|\,.
\end{align*}
With the similar argument as in the proof of Lemma \ref{Key_lemma}, we have $|\omega_f^{(g)}(t,\xi,\lambda) - \phi_k^{\prime}(t)| \leq \Tilde{\epsilon}$, if $\epsilon$ is sufficiently small.
\end{proof}
Combining all these lemmas, one can conclude the proof of Theorem \ref{Main_theorem}.

 \bibliographystyle{elsarticle-num} 
 \bibliography{main}





\end{document}